\documentclass[11pt,a4paper]{article}

\usepackage[top=2.5cm, bottom=2.5cm, left=3cm, right=2.5cm]{geometry}
\usepackage{graphicx}%
\usepackage{multirow}%
\usepackage{amsmath,amssymb,amsfonts}%
\usepackage{amsthm}%
\usepackage{mathrsfs}%
\usepackage[title]{appendix}%
\usepackage{xcolor}%
\usepackage{textcomp}%
\usepackage{manyfoot}%
\usepackage{booktabs}%
\usepackage{algorithm}%
\usepackage{algorithmicx}%
\usepackage{algpseudocode}%
\usepackage{listings}%

\newtheorem{theorem}{Theorem}
\newtheorem{proposition}[theorem]{Proposition}%
\newtheorem{remark}{Remark}%
\newtheorem{lemma}{Lemma}%
\newtheorem{corollary}{Corollary}%

\newcommand{\tp}{^\top}
\newcommand{\bb}[1]{\mathbb{#1}}
\newcommand{\ssp}[1]{{\left( #1 \right)}}
\newcommand{\mmp}[1]{{\left[ #1 \right]}}

\newcommand{\nnp}[1]{{\left\| #1 \right\|}}
\newcommand{\inner}[1]{{\left\langle #1 \right\rangle}}
\newcommand{\tr}{\mathrm{tr}}
\newcommand{\Diag}{\mathrm{Diag}}
\usepackage{subcaption}
\usepackage{hyperref}

\title{Exact Quadratic Penalty Function for Symplectic Eigenvalue Problem}

\author{Jiaqi Wang\thanks{State Key Laboratory of Scientific and Engineering Computing, Academy of Mathematics and Systems Science, Chinese Academy of Sciences, and University of Chinese Academy of Sciences, China \url{wangjiaqi@lsec.cc.ac.cn}} \and Nachuan Xiao\thanks{School of Data Science, The Chinese University of Hong Kong, Shenzhen, Guangdong, China \url{xncxy@cuhk.edu.cn}} \and Xin Liu\thanks{State Key Laboratory of Scientific and Engineering Computing, Academy of Mathematics and Systems Science, Chinese Academy of Sciences, and University of Chinese Academy of Sciences, China \url{liuxin@lsec.cc.ac.cn}}}

\date{\today}

\begin{document}

\maketitle

\begin{abstract}
The symplectic eigenvalue problem for symmetric positive-definite (SPD) matrices plays a crucial role in various scientific fields, including quantum mechanics and control theory. While this problem can be formulated as an optimization problem over the symplectic Stiefel manifold, existing Riemannian optimization frameworks admit high computational costs due to expensive geometric complications. Furthermore, the non-compactness of the manifold imposes challenges in theoretical analysis, as it fails to guarantee the uniform non-singularity of its associated Jacobian. In this paper, we transform the symplectic eigenvalue problem into an unconstrained minimization of the trace-penalty function. This reformulation enables the direct application of various unconstrained optimization methods. We establish the equivalence between the symplectic eigenvalue problem and the proposed trace-penalty function under mild assumptions. Numerical experiments demonstrate that the proposed algorithm significantly outperforms a wide range of existing approaches, including Riemannian optimization methods, in terms of computational efficiency. These results highlight the substantial potential of our algorithm, particularly for large-scale symplectic eigenvalue problems.
\end{abstract}

\section{Introduction}
Given a positive integer $k$, we consider the matrix $J_k \in \bb{R}^{2k \times 2k}$ as 
\begin{equation*}
J_k = \begin{bmatrix}
0 & I_k \\
-I_k & 0
\end{bmatrix},
\end{equation*}
where $I_k$ denotes the $k$-th order identity matrix. The matrix $J_k$ is also known as a Poisson matrix \cite{Peng2016Symplectic}. A matrix $S \in \mathbb{R}^{2n \times 2p}$ is called a symplectic matrix \cite{Peng2016Symplectic, Gao2021Riemannian} if 
\begin{equation*}
S^\top J_n S = J_p.
\end{equation*}
Although the term \textit{symplectic} has historically been restricted to square matrices, its definition has recently been extended to rectangular matrices \cite{Peng2016Symplectic, Gao2021Riemannian}. 

In \cite{Williamson1936Algebraic}, Williamson proved that for any symmetric positive definite (SPD) matrix $A \in \mathbb{R}^{2n \times 2n}$, there exists a symplectic matrix $S \in \mathbb{R}^{2n \times 2n}$ that diagonalizes $A$ in the sense that 
\begin{equation}\label{eq:williamson}
S^\top A S = \begin{bmatrix}
D & 0 \\
0 & D
\end{bmatrix},
\end{equation}
where $D = \mathrm{Diag}\ssp{d_1, \dots, d_n}$ with $d_j > 0$ for $j = 1, \dots, n$. The right-hand side of \eqref{eq:williamson} is called Williamson's diagonal form, and $d_1, \dots, d_n$ are referred to as the symplectic eigenvalues. Moreover, the columns of $S$ are called the symplectic eigenvectors, in the sense that 
\begin{equation}
    A [S_i, S_{i+n}] = d_i J_n [S_i, S_{i+n}] J_1\tp, \quad \forall~ 1\leq i \leq n
\end{equation}
Here $S_i$ refers to the $i$-th column vector of $S$. 

Symplectic eigenvalues have extensive applications in diverse fields, including optics \cite{Eisert2008Gaussian}, quantum mechanics \cite{Krbek2014Inequalities}, quantum computing \cite{Banchi2015Quantum}, quantum information and communication \cite{Adesso2007Entanglement, Weedbrook2012Gaussian, Chakraborty2025Hybrid}, learning of linear port-Hamiltonian systems via their normal forms \cite{Ortega2024Learnability}, and stability analysis of weakly damped gyroscopic systems \cite{Benner2005Skew}. For example, in quantum mechanics, they play a crucial role in characterizing quantum systems and their subsystems with Gaussian states \cite{Hiroshima2006Additivity, Parthasarathy2013Symmetry, Krbek2014Inequalities}. Within the Gaussian marginal problem, knowledge of the symplectic eigenvalues facilitates the determination of local entropies compatible with a given joint state \cite{Eisert2008Gaussian}.

The computation of standard eigenvalues constitutes a well-established subfield in numerical linear algebra (see \cite{Kressner2005Numerical, Watkins2007Matrix, Saad2011Numerical}). In particular, optimization-based numerical methods have been extensively investigated. These methods involve minimizing either the matrix trace or the Rayleigh quotient under specific constraints. However, relatively few studies have specifically focused on computing symplectic eigenvalues. Some constructive proofs, as shown in \cite{Simon1999Congruences, Parthasarathy2013Symmetry}, result in numerical methods that are only appropriate for problems of small to medium size.

It is important to note that \eqref{eq:williamson} constitutes a special case of the results presented in \cite{Williamson1936Algebraic}, where the authors aimed to determine the canonical forms of the matrix pencil $A - \lambda B$ with $A$ symmetric and $B$ skew-symmetric. Subsequently, several works, including the review \cite{Lancaster2005Canonical} and the more recent study \cite{Bovdi2018Symplectic}, have revisited this topic. These works indicate that the problem of computing the symplectic eigenvalues of a $2n \times 2n$ SPD matrix $A$ can be reformulated as the (standard) eigenvalue problem of the skew-symmetric matrix $A^{1/2} J_n A^{1/2}$ \cite{Jain2022Derivatives}, Hermitian matrix pencil $A - i\lambda J_n$ \cite{Lancaster2006Canonical}, or positive-definite Hamiltonian (pdH) matrix $J_n A$ \cite{Amodio2003Symplectic}. Here, $A^{1/2}$ denotes the symmetric positive-definite square root of $A$. Consequently, properties such as the linear independence of symplectic eigenvectors corresponding to distinct symplectic eigenvalues, the uniqueness of such decompositions, and even certain constructive proofs of Williamson's theorem \cite{Simon1999Congruences, Parthasarathy2013Symmetry, Ikramov2018Symplectic} can be verified using these facts. However, none of these cases yield a symmetric matrix for the (standard) eigenvalue problem, thus failing to provide a simplified problem.

The approaches in \cite{Amodio2006Computation, Ikramov2018Symplectic} establish a one-to-one correspondence between SPD and positive-definite Hermitian matrices. In particular, \cite{Ikramov2018Symplectic} demonstrates that the symplectic eigenvalues of $A$ can be determined by transforming the pdH matrix $J_n A$ into a normal form via elementary symplectic transformations, as outlined in \cite{Ikramov1991Conditions}. This relationship is also utilized in the symplectic Lanczos method to find several extreme eigenvalues of the pdH matrices, as presented in \cite{Amodio2003Symplectic, Amodio2006Computation}. Additionally, in \cite{Gruning2011Implementation, Penke2020High}, Lanczos-type algorithms have been developed to compute the eigenvalues of the skew-symmetric matrix $A^{1/2}J_nA^{1/2}$, a task linked to determining the symplectic eigenvalues of $A$ as noted earlier.

\subsection{Solving symplectic eigenvalue problem through constrained optimization}
To efficiently solve the symplectic eigenvalue problem, \cite{Son2021Computing} considers the following constrained optimization problem, 
\begin{equation}
\label{Prob_Ori}
\tag{COP}
\begin{aligned}
\min_{X \in \mathbb{R}^{2n \times 2p}} \quad &\frac{1}{2} \langle X, AX \rangle \\
\text{s. t.} \quad & X^\top J_n X = J_p.
\end{aligned}
\end{equation}
Here, the inner product $\langle X, Y \rangle$ is defined as $\mathrm{tr}\left(X^\top Y\right)$. The feasible set of \eqref{Prob_Ori} is a Riemannian manifold, denoted by
\begin{equation*}
\mathrm{Sp}(2p, 2n) := \left\{ X \in \mathbb{R}^{2n \times 2p} \mid X^\top J_n X = J_p \right\},
\end{equation*}
commonly referred to as the \textit{symplectic Stiefel manifold} when $p < n$. The Riemannian manifold $\mathrm{Sp}(2p, 2n)$ degenerates into a symplectic group when $p = n$ and is denoted by $\mathrm{Sp}(2n)$ for simplicity.

The optimization problem \eqref{Prob_Ori} is closely related to the symplectic eigenvalue problem for a symmetric positive-definite (SPD) matrix $A \in \mathbb{R}^{2n \times 2n}$. Using the trace minimization theorem, which is discussed in detail in \cite{Hiroshima2006Additivity, Bhatia2015Symplectic}, the smallest $k$ ($1 \leq k \leq n$) symplectic eigenvalues of a given SPD matrix $A$ can be reduced to minimize the trace cost function under the symplecticity constraint, that is, the optimization problem \eqref{Prob_Ori}. The trace minimization theorem states that the global minimum value of \eqref{Prob_Ori} is equal to the sum of the smallest $k$ symplectic eigenvalues of $A$. Once the global minimizer $X_{\mathrm{min}}$ is obtained, let symplectic matrix $X_{\mathrm{diag}}$ diagonalize $X_{\mathrm{min}}^\top A X_{\mathrm{min}}$. The resulting diagonal elements are the corresponding symplectic eigenvalues, and $X_{\mathrm{min}} X_{\mathrm{diag}}$ are the associated symplectic eigenvectors. In other words, the solution to the trace minimization problem provides only a symplectic basis for the subspace corresponding to the smallest $k$ symplectic eigenvalues.

It is worth noting that the symplectic Stiefel manifold $\text{Sp}(2p,2n)$ is a Riemannian manifold, hence existing first-order algorithms for Riemannian optimization \cite{Absil2009Optimization} can be applied once the mechanisms for determining the search directions and retraction are defined. From this perspective, for general optimization problems with symplectic constraints, a steepest descent method has been proposed in \cite{Gao2021Riemannian}, incorporating two types of retractions on $\text{Sp}(2p,2n)$: the quasi-geodesic retraction and the Cayley retraction. Subsequently, \cite{Oviedo2023Collection, Oviedo2021Efficient} introduced a more efficient retraction, exhibiting numerical and empirical superiority over the Cayley retraction in \cite{Gao2021Riemannian}. For the case $n = p$, Newton’s method was proposed in \cite{Birtea2020Optimization} as an optimization method for symplectic groups; however, it cannot be applied straightforwardly when $n > p$. For problem \eqref{Prob_Ori}, \cite{Son2021Computing} constructed an algorithm by leveraging the Riemannian structure of the feasible set investigated in \cite{Gao2021Riemannian}. As noted earlier, solving \eqref{Prob_Ori} yields only a symplectic basis for the eigensubspace rather than eigenvectors. Recently, \cite{Son2025Brockett} introduces a Brockett cost function, revealing connections to symplectic eigenvalues/eigenvectors, and specifically demonstrates that any critical point comprises symplectic eigenvectors. 

Although \eqref{Prob_Ori} is equivalent to the symplectic eigenvalue problem, developing efficient optimization algorithms for \eqref{Prob_Ori} remains a challenging task. The nonconvexity of the constraints of \eqref{Prob_Ori} prevents the direct implementations of various efficient solvers for unconstrained optimization. Moreover, although Riemannian optimization methods are developed based on their unconstrained counterparts, these methods require the computation of the related geometric materials, including the Riemannian gradients, retractions, vector transports, etc. However, computing these geometric materials requires the solution of Lyapunov equations, hence is computationally expensive in practice, especially in large-scale optimization problems. 


\subsection{Contributions}
To develop efficient optimization methods for solving \eqref{Prob_Ori}, we consider to reformulate \eqref{Prob_Ori} into an unconstrained optimization problem. It is worth noting that for the standard eigenvalue problems,  \cite{Wen2016TracePenalty} proposed an unconstrained trace-penalty minimization model based on a quadratic penalty function. Inspired by this work, we adapt the quadratic penalty functions to the symplectic eigenvalue problem and consider the following optimization problem,
\begin{equation}
\tag{TPM}
\label{Prob_Pen}
\min_{X \in \mathbb{R}^{2n \times 2p}} \quad f_\beta(X) := \frac{1}{2} \langle X, AX \rangle + \frac{\beta}{4} \nnp{X\tp J_n X - J_p}_F^2,
\end{equation}
where $\beta > 0$ is a penalty parameter. Notably, quadratic penalty functions are exact for standard eigenvalue problems \cite{Wen2016TracePenalty}, as the compactness of the constraint $X\tp X = I_p$ restricts the equivalence analysis to a uniformly bounded set and the Jacobian of the mapping $X \mapsto X\tp X - I_p$ is uniformly non-singular (i.e., the smallest singular value is uniformly bounded away from $0$) over that uniformly bounded set. Conversely, the non-compactness of the symplectic constraint $X\tp J_n X = J_p$ in \eqref{Prob_Ori} requires us to establish the equivalence analysis over an unbounded set, where the Jacobian of $X\mapsto X\tp J_n X - J_p$ is possibly not uniformly non-singular. This difference introduces significant theoretical challenges in establishing the relationship between \eqref{Prob_Ori} and \eqref{Prob_Pen}.

In this paper, we establish the equivalence between the penalty problem \eqref{Prob_Pen} and the original problem \eqref{Prob_Ori} in terms of eigensubspaces. Specifically, we prove that the global minimizers of the penalty problem preserve the subspace of the global minimizers of the original problem. Furthermore, we derive an explicit expression for all first-order stationary points of the penalty problem and demonstrate that no local minimizer exists for the penalty problem. It is worth mentioning that evaluating the gradient of \eqref{Prob_Pen} only involves  matrix-matrix multiplications, hence avoiding the computationally expensive geometric computations of the symplectic Stiefel manifold.



Furthermore, we develop a first-order method for solving the penalty problem \eqref{Prob_Pen}. The proposed method is a non-monotonic gradient method that incorporates the Barzilai-Borwein (BB) step size, which is designed to solve this unconstrained problem.

Preliminary numerical experiments demonstrate the effectiveness of the proposed algorithm, in the aspects of accuracy of the yielded subspaces, residual of the optimal solution, and wall-clock time, in comparison with existing algorithms. These numerical results also demonstrate that we can develop efficient algorithms based on our proposed trace-penalty function and the techniques from unconstrained optimization.

\subsection{Organization}
The remainder of this paper is organized as follows. We analyze the trace-penalty minimization model in \ref{sec:model}. Our algorithms and several implementation details are discussed in \ref{sec:alg}. The numerical results are reported in \ref{sec:numexp}. Finally, we conclude the paper in \ref{sec:con}.

\section{Exactness of the Trace Penalty Function}
\label{sec:model}
In this section, we first prove the equivalence between \eqref{Prob_Ori} and \eqref{Prob_Pen} by demonstrating that the global minimizers of these two problems span the same eigensubspace. Second, the penalty problem \eqref{Prob_Pen} is proven to have no local minimizer. Finally, the condition numbers of the Hessians at the global minimizers are estimated.

\subsection{Preliminaries and Notations}

In this paper, let $0_{p\times q}$ denote the all-zero matrix with size $p\times q$. For square matrix $A,B$, let $A\oplus B := \begin{bmatrix}A & \\ & B\end{bmatrix}$. $\|\cdot\|_2$ and $\|\cdot\|_F$ represent the $L_2$-norm and the Frobenius norm of a matrix, respectively. For a positive semi-definite matrix, $A^{1/2}$ refers to the square root of $A$, namely the unique positive semi-definite matrix satisfying $A^{1/2}\cdot A^{1/2} = A$. Let $\mathrm{Skew}(2k)$ denote the $2k\times 2k$ skew-symmetric matrix set. The $k\times k$ orthonormal matrix set is abbreviated by $\mathrm{Orth}(k)$, and $\mathrm{OrSp}(2k)$ refers to the set of $2k\times 2k$ matrices that are not only symplectic but also orthonormal.

Additionally, a matrix $W$ is called \textit{skew-Hamiltonian} if $WJ_n=-\ssp{WJ_n}\tp$ \cite{Benner2005Skew}. We present the following theorem that illustrating the properties of skew-Hamiltonian matrices that is essential in our theoretical analysis. 
\begin{theorem}[\cite{VanLoan1984}]
\label{thm:PVL}
Given a skew-Hamiltonian matrix $W\in\bb{R}^{2n\times 2n}$, there is always an orthogonal symplectic matrix $U$ so that $U\tp WU$ has Paige/Van Loan (PVL) form, i.e.
\begin{equation*}
U\tp WU =
\begin{bmatrix}
W_{11} & W_{12}\\
& W_{11}\tp
\end{bmatrix}
\end{equation*}
where $W_{11}\in\bb{R}^{n\times n}$ is an upper Hessenberg matrix with diagonal blocks of order one and two corresponding, respectively, to real and complex standard eigenvalues of $W$.
\end{theorem}

In this paper, we adopt the definitions from \cite{nocedal2006numerical} on the optimality conditions for the constrained problem \eqref{Prob_Ori}. In particular, we say $X \in \bb{R}^{2p\times 2p}$ is a first-order  stationary point of \eqref{Prob_Ori}, if there exists $L\in \bb{R}^{2p\times 2p}$ such that
\begin{equation*}
\left\{
\begin{aligned}
& AX = J_n X L,\\
& X\tp J_n X =J_p.
\end{aligned}
\right.
\end{equation*}

Moreover, for the unconstrained optimization problem \eqref{Prob_Pen}, we say $X \in \bb{R}^{2p\times 2p}$ is a first-order stationary point of \eqref{Prob_Pen}, if $\nabla h(X) = 0$. 
Moreover, we say $X \in \bb{R}^{2p\times 2p}$ is a second-order stationary point of \eqref{Prob_Pen}, if $X$ is a first-order stationary point of \eqref{Prob_Pen}, and $\nabla^2 h(X) \succeq 0$.

\subsection{Equivalence Between Stationary Points}
We define equivalence between the constrained and unconstrained problems as follows: \eqref{Prob_Pen} is \textit{equivalent} to \eqref{Prob_Ori} if their global minimizers span identical $2p$-dimensional eigenspaces corresponding to the $p$ smallest symplectic eigenvalues of $A$. Our main theoretical result shows that this equivalence holds when the penalty parameter satisfies $\beta > d_p$, where $d_p$ denotes the largest target symplectic eigenvalue. This condition ensures $\beta$ need not be excessively large, thereby avoiding numerical ill-conditioning.

Let Williamson's diagonal form of $A$ be uniquely expressed as:
\begin{equation*}
D := \mathrm{Diag}\left(d_1, d_2, \dots, d_n, d_1, d_2, \dots, d_n\right)
\end{equation*}
with $0 < d_1 \leq d_2 \leq \dots \leq d_n$, and let $S\in\mathrm{Sp}(2n)$ be one of the symplectic matrices satisfying $S^\top A S = D$.

\begin{proposition}
\label{prop:SPD-skew-Ham}
For any $\beta > 0$, if $X \in \mathbb{R}^{2n \times 2p}$ is a first-order stationary point of \eqref{Prob_Pen}, then the matrix $-X^\top J_n X J_p$ is symmetric positive semi-definite and skew-Hamiltonian. Furthermore, its spectral norm satisfies $\| -X^\top J_n X J_p \|_2 \leq 1$.
\end{proposition}

\begin{proof}
From the first-order optimality condition $\nabla f_\beta(X) = 0$, we have
\begin{equation}
\label{eq:grad-fbeta-X-0}
AX + \beta J_n X J_p (X^\top J_n X J_p)^\top + \beta J_n X J_p = 0.
\end{equation}
Define $W := -X^\top J_n X J_p$. Multiplying \eqref{eq:grad-fbeta-X-0} on the left by $X^\top$ yields $X^\top AX + \beta WW^\top - \beta W = 0$. 
Rearranging terms, we obtain $W = \frac{1}{\beta} X^\top AX + WW^\top$. 
This expression shows that $W$ is symmetric positive semi-definite. Moreover, the inequality $W - W^2 = \frac{1}{\beta} X^\top AX \succeq 0$ implies $\|W\|_2 \leq 1$ by the spectral theorem.

To verify the skew-Hamiltonian property, observe that
\begin{equation*}
W J_p = X^\top J_n X = -(X^\top J_n X)^\top = -(W J_p)^\top.
\end{equation*}
This confirms that $W$ is skew-Hamiltonian, completing the proof.
\end{proof}

\begin{lemma}
\label{lem:TDT}
If $W \in \mathbb{R}^{2p \times 2p}$ is both symmetric and skew-Hamiltonian, then there exists an orthogonal symplectic matrix $T \in \mathrm{OrSp}(2p)$ that diagonalizes $W$ into the form
\begin{equation*}
T^\top W T = D \oplus D,
\end{equation*}
where $D \in \mathbb{R}^{p \times p}$ is a diagonal matrix.
\end{lemma}

\begin{proof}
By Theorem \ref{thm:PVL}, there exists an orthogonal symplectic matrix $T \in \mathrm{OrSp}(2p)$ such that
\begin{equation*}
T^\top W T = \begin{bmatrix}
W_{11} & W_{12} \\
0 & W_{11}^\top
\end{bmatrix}.
\end{equation*}
Since $W$ is symmetric, we have
\begin{equation*}
\begin{bmatrix}
W_{11} & W_{12} \\
0 & W_{11}^\top
\end{bmatrix}
=
\begin{bmatrix}
W_{11}^\top & 0 \\
W_{12}^\top & W_{11}
\end{bmatrix}.
\end{equation*}
Equating the blocks, it follows that $W_{12} = 0$ and $W_{11} = W_{11}^\top$. Thus, $W_{11}$ is symmetric and therefore diagonalizable. Let $D = W_{11}$, and the result follows.
\end{proof}

The following proposition characterizes the structure of first-order stationary points for \eqref{Prob_Pen}, showing they can be decomposed in terms of symplectic eigenvectors of $A$. 
\begin{proposition}
\label{prop:stpform}
For any $\beta > 0$, a matrix $X \in \mathbb{R}^{2n \times 2p}$ is a first-order stationary point of \eqref{Prob_Pen} if and only if it admits the decomposition:
\begin{equation*}
X = \mmp{\hat{S}_1 \left(I_q - \hat{D}/\beta\right)^{1/2} \;\; 0_{2n \times (p-q)} \;\; \hat{S}_2 \left(I_q - \hat{D}/\beta\right)^{1/2} \;\; 0_{2n \times (p-q)}} T^\top,
\end{equation*}
where $q \leq p$, $\hat{S} = \begin{bmatrix}\hat{S}_1 & \hat{S}_2\end{bmatrix} \in \mathrm{Sp}(2q, 2n)$ consists of $q$ symplectic eigenvector pairs of $A$, diagonalizing $A$ in the sense
\begin{equation*}
\hat{S}^\top A \hat{S} = \hat{D} \oplus \hat{D},
\end{equation*}
with $\beta I_q \succ \hat{D} \succ 0$, and $T \in \mathrm{OrSp}(2p)$.
\end{proposition}

\begin{proof}
We first prove necessity. Assume $X \in \mathbb{R}^{2n \times 2p}$ is a stationary point of \eqref{Prob_Pen}. By Proposition \ref{prop:SPD-skew-Ham} and Lemma \ref{lem:TDT}, there exist $T \in \mathrm{OrSp}(2p)$ and a diagonal matrix $\hat{\Sigma} = \Diag(\sigma_1, \dots, \sigma_q) \in \mathbb{R}^{q \times q}$ with $1 \geq \sigma_1 \geq \sigma_2 \geq \dots \geq \sigma_q > 0$ such that
\begin{equation}
\label{eq:stationary-W}
-X^\top J_n X J_p = T \ssp{\hat{\Sigma} \oplus 0_{(p-q) \times (p-q)} \oplus \hat{\Sigma} \oplus 0_{(p-q) \times (p-q)}} T^\top.
\end{equation}
Multiplying \eqref{eq:stationary-W} on the right by the matrix $T \ssp{\hat{\Sigma}^{-1/2} \oplus I_{p-q} \oplus \hat{\Sigma}^{-1/2} \oplus I_{p-q}}$, and on the left by the matrix $\ssp{\hat{\Sigma}^{-1/2} \oplus I_{p-q} \oplus \hat{\Sigma}^{-1/2} \oplus I_{p-q}} T^\top$ yields
\begin{equation}
\begin{aligned}
&\left(X T \ssp{\hat{\Sigma} \oplus I_{p-q} \oplus \hat{\Sigma} \oplus I_{p-q}}^{-1/2}\right)^\top J_n \left(X T \ssp{\hat{\Sigma} \oplus I_{p-q} \oplus \hat{\Sigma} \oplus I_{p-q}}^{-1/2}\right) \\
={}& J_p \left(I_q \oplus 0_{p-q} \oplus I_q \oplus 0_{p-q}\right). 
\end{aligned}
\end{equation}
Define
\begin{equation}
\tilde{S} := X T \ssp{\hat{\Sigma} \oplus I_{p-q} \oplus \hat{\Sigma} \oplus I_{p-q}}^{-1/2} = \mmp{\hat{S}_1 \;\; Y_1 \;\; \hat{S}_2 \;\; Y_2}.
\end{equation}
Then let $\hat{S} := \mmp{\hat{S}_1 \;\; \hat{S}_2} \in \mathrm{Sp}(2q, 2n)$, and $X$ can be expressed as
\begin{equation*}
X = \tilde{S} \ssp{\hat{\Sigma} \oplus I_{p-q} \oplus \hat{\Sigma} \oplus I_{p-q}}^{1/2} T^\top.
\end{equation*}
Let $\tilde{\Sigma} := \hat{\Sigma} \oplus I_{p-q} \oplus \hat{\Sigma} \oplus I_{p-q}$. Substituting into the stationary condition $\nabla f_\beta(X) = AX - \beta J_n X (X^\top J_n X - J_p) = 0$, we have
\begin{equation}
\begin{aligned}
\beta J_n X \left(X^\top J_n X - J_p\right) = -\beta J_n \tilde{S} \tilde{\Sigma}^{1/2} \left(I_{2p} - \tilde{\Sigma}\right) T^\top, \quad \text{and} \quad 
AX = A \tilde{S} \tilde{\Sigma}^{1/2} T^\top.
\end{aligned}
\end{equation}
Equating terms gives $A \tilde{S} = J_n \tilde{S} J_p \cdot \beta \left(\tilde{\Sigma} - I_{2p}\right)$, 
which leads to $A \hat{S} = J_n \hat{S} J_p \cdot \beta \left(\hat{\Sigma} \oplus \hat{\Sigma} - I_{2q}\right)$ and $A \mmp{Y_1 \;\; Y_2} = 0_{2n \times 2(p-q)}$. Let $\hat{D} := \beta (\hat{\Sigma} - I_q)$, whose diagonal elements are the $q$ symplectic eigenvalues of $A$, with $\beta I_q \succ \hat{D} \succ 0$. Then we have
\begin{equation}
X = \mmp{\hat{S}_1 \left(I_q - \hat{D}/\beta\right)^{1/2} \;\; 0_{2n \times (p-q)} \;\; \hat{S}_2 \left(I_q - \hat{D}/\beta\right)^{1/2} \;\; 0_{2n \times (p-q)}} T^\top.
\end{equation}
This completes the first part of the proof. 

For sufficiency, let
\begin{equation}
X = \mmp{\hat{S}_1 \left(I_q - \hat{D}/\beta\right)^{1/2} \;\; 0_{2n \times (p-q)} \;\; \hat{S}_2 \left(I_q - \hat{D}/\beta\right)^{1/2} \;\; 0_{2n \times (p-q)}} T^\top.
\end{equation}
Substituting into $\nabla f_\beta(X) = 0$ and simplifying verifies that the condition holds. Thus, $X$ is a first-order stationary point of \eqref{Prob_Pen}. This completes the entire proof.
\end{proof}

Then, the following proposition establishes an injective correspondence between the sets of full-rank first-order stationary points of the penalized problem \eqref{Prob_Pen} and the original problem \eqref{Prob_Ori}.
\begin{proposition}
\label{prop:corr}
Let $\mathcal{S}_{\mathrm{tpm}}$ denote the set of full-rank first-order stationary points of \eqref{Prob_Pen}, and let $\mathcal{S}_{\mathrm{cop}}$ denote the set of first-order stationary points of \eqref{Prob_Ori}. The mapping
\begin{equation*}
\mathcal{T} : \mathcal{S}_{\mathrm{tpm}} \rightarrow \mathcal{S}_{\mathrm{cop}}, \quad \hat{S}(I_{2p} - \hat{D}/\beta)^{1/2}T^\top \mapsto \hat{S}T^\top,
\end{equation*}
is well-defined and injective.
\end{proposition}
\begin{proof}
To show that $\mathcal{T}$ is well-defined, note that any full-rank stationary point $X \in \mathcal{S}_{\mathrm{tpm}}$ admits the unique decomposition $X = \hat{S}(I_{2p} - \hat{D}/\beta)^{1/2}T^\top$, where $\hat{S}$ corresponds to the columns of the Williamson diagonalizing matrix $S$. Thus, the mapping $\mathcal{T}(X) = \hat{S}T^\top$ is uniquely determined.

For injectivity, consider two distinct stationary points $X, Y \in \mathcal{S}_{\mathrm{tpm}}$ such that $X \neq Y$. If $\mathcal{T}(X) = \mathcal{T}(Y)$, then $\hat{S}_X T_X^\top = \hat{S}_Y T_Y^\top$. This implies either $\hat{S}_X \neq \hat{S}_Y$ (distinct symplectic components) or $T_X \neq T_Y$ (distinct orthogonal transformations), both of which contradict the assumption that $X \neq Y$. Therefore, $\mathcal{T}(X) \neq \mathcal{T}(Y)$, proving injectivity.

Therefore, we can conclude that $\mathcal{T}$ is well-defined and injective. This completes the proof. 
\end{proof}

Proposition \ref{prop:corr} establishes that the number of full-rank first-order stationary points in the penalized problem \eqref{Prob_Pen} does not exceed that of the original constrained problem \eqref{Prob_Ori}. Consequently, by selecting appropriate penalty parameters, the proposed unconstrained approach identifies the optimal symplectic eigenspace with fewer spurious saddle points compared to the original constrained formulation.

Next, the following lemma demonstrates that every matrix $X \in \mathbb{R}^{2n \times 2p}$ with full symplectic rank admits a symplectic singular value decomposition (SSVD).
\begin{lemma}
\label{lem:sym-svd}
Assume $X \in \mathbb{R}^{2n \times 2p}$ satisfies $X\tp J_nX$ is invertible, then $X$ admits a symplectic singular value decomposition (SSVD) of the form
\begin{equation*}
X = S \Sigma T^\top,
\end{equation*}
where $S \in \mathrm{Sp}(2p, 2n)$, $T \in \mathrm{Orth}(2p)$, and $\Sigma = \Diag(\sigma_1, \dots, \sigma_p, \sigma_1, \dots, \sigma_p)$ with $0 < \sigma_1 \leq \dots \leq \sigma_p$.
\end{lemma}

\begin{proof}
Since $X^\top J_n X$ is skew-symmetric, it admits a real Schur decomposition of the form
\begin{equation*}
X\tp J_nX = Q \ssp{\begin{bmatrix}0 & d_1\\ -d_1 & 0\end{bmatrix}\oplus\dots\oplus\begin{bmatrix}0 & d_p\\ -d_p & 0\end{bmatrix}} Q\tp.
\end{equation*}
where $Q \in \mathrm{Orth}(2p)$ and $d_1, \dots, d_p > 0$. Define the permutation matrix
\begin{equation*}
P := \begin{bmatrix} e_1, e_3, \dots, e_{2p-1}, e_2, e_4, \dots, e_{2p} \end{bmatrix}
\end{equation*}
to rearrange components, yielding
\begin{equation*}
X^\top J_n X = Q P J_p \Diag(D, D) P^\top Q^\top,
\end{equation*}
where $D := \Diag(d_1, \dots, d_p)$.

Next, multiply $QP \Diag(D, D)^{-1/2}$ on the right and $\Diag(D, D)^{-1/2} P^\top Q^\top$ on the left to ensure
\begin{equation*}
\left(X QP \Diag(D, D)^{-1/2}\right)^\top J_n \left(X QP \Diag(D, D)^{-1/2}\right) = J_p.
\end{equation*}
Define
\begin{equation*}
S := X QP \Diag(D, D)^{-1/2},
\end{equation*}
which satisfies $S \in \mathrm{Sp}(2p, 2n)$. Thus, we can write
\begin{equation*}
X = S \Diag(D, D)^{1/2} P^\top Q^\top := S \Sigma T^\top,
\end{equation*}
where $\Sigma := \Diag(D, D)^{1/2}$ and $T := QP \in \mathrm{Orth}(2p)$.

This completes the proof.
\end{proof}

\begin{remark}
For the stationary points of \eqref{Prob_Pen}, full column rank implies full symplectic rank, i.e. $X\tp J_nX$ is invertible, by Proposition \ref{prop:stpform}. Thus, any full-rank stationary point admits an SSVD by Lemma \ref{lem:sym-svd}.
\end{remark}

The decomposition in Lemma \ref{lem:sym-svd} enables the extraction of symplectic eigenvalues and eigenvectors from any full-rank minimizer $X_{\mathrm{fin}}$ of \eqref{Prob_Pen} via the Symplectic Rayleigh-Ritz (SRR) procedure, as follows:
\begin{enumerate}
    \item Compute the SSVD: $X_{\mathrm{fin}} = S\Sigma T^\top$.
    \item Find $\hat{S} \in \mathrm{Sp}(2p)$ such that $\hat{S}^\top (S^\top A S) \hat{S} = D_{\mathrm{fin}}$.
    \item Set $S_{\mathrm{fin}} := S\hat{S}$.
\end{enumerate}
We denote this process as $(S_{\mathrm{fin}}, D_{\mathrm{fin}}) = \mathrm{SRR}(X_{\mathrm{fin}})$.

Then, we present the following proposition, which establishes an upper bound for $\|AS - J_n SL\|_F$.

\begin{proposition}
\label{prop:corr2}
For any matrix $X \in \mathbb{R}^{2n \times 2p}$ satisfying $X\tp J_nX$ is invertible with SSVD $X = S\Sigma T^\top$, there exists $L \in \mathbb{R}^{2p \times 2p}$ such that
\begin{equation*}
\|AS - J_n SL\|_F \leq \sqrt{\frac{2p d_n}{\sigma_{\min}(X^\top AX)}} \cdot \|\nabla f_\beta(X)\|_F.
\end{equation*}
\end{proposition}

\begin{proof}
By the trace minimization theorem (\cite{Hiroshima2006Additivity, Bhatia2015Symplectic}), for any $S \in \mathrm{Sp}(2p, 2n)$, $\tr(S^\top AS)$ is no larger than twice the sum of the $p$ largest symplectic eigenvalues. Hence, we have
\begin{equation*}
2p d_n \geq \tr(S^\top AS) = \tr(\Sigma^{-1} T^\top X^\top AX T \Sigma^{-1}) \geq \sigma_{\min}(X^\top AX) \cdot \|\Sigma^{-1}\|_F^2.
\end{equation*}

Next, consider the gradient expression for $X = S\Sigma T^\top$. Using the definition of $\nabla f_\beta(X)$, we compute:
\begin{equation*}
\begin{aligned}
&\nabla f_\beta(X) = AX - \beta J_n X \left(X^\top J_n X - J_p\right) = AS\Sigma T^\top - \beta J_n S\Sigma T^\top \left(X^\top J_n X - J_p\right) \\
={}& AS\Sigma T^\top - \beta J_n S\Sigma T^\top \left(T\Sigma^2 J_p T^\top - J_p\right) =AS\Sigma T^\top - \beta J_n S \left(\Sigma^3 J_p T^\top - \Sigma T^\top J_p\right).
\end{aligned}
\end{equation*}
This illustrates that 
\begin{equation*}
\|AS - \beta J_n S \left(\Sigma^3 J_p T^\top - \Sigma T^\top J_p\right) T \Sigma^{-1}\|_F = \|\nabla f_\beta(X) T \Sigma^{-1}\|_F.
\end{equation*}
Applying the submultiplicative property of the Frobenius norm, we obtain:
\begin{equation*}
\begin{aligned}
&\|\nabla f_\beta(X) T \Sigma^{-1}\|_F \leq \|T \Sigma^{-1}\|_F \cdot \|\nabla f_\beta(X)\|_F \\
={}& \|\Sigma^{-1}\|_F \cdot \|\nabla f_\beta(X)\|_F  \leq  \sqrt{\frac{2p d_n}{\sigma_{\min}(X^\top AX)}} \cdot \|\nabla f_\beta(X)\|_F.
\end{aligned}
\end{equation*}
Therefore, we complete the proof with $L := \beta \left(\Sigma^3 J_p T^\top - \Sigma T^\top J_p\right) T \Sigma^{-1}$.
\end{proof}

Proposition \ref{prop:corr2} establishes that, through the SRR procedure, an $\varepsilon$-stationary point of the penalized problem \eqref{Prob_Pen} yields a $c(X)\varepsilon$-stationary point of the original problem \eqref{Prob_Ori}. Here, $c(X)=\sqrt{\frac{2p d_n}{\sigma_{\min}(X^\top AX)}}$ is a constant depending on the nonsingularity of $X$. Furthermore, when $\varepsilon$ is sufficiently small, $c(X)$ admits a uniform upper bound.

Our central equivalence result establishes that the penalty formulation precisely captures the solutions of the original constrained problem when the penalty parameter is appropriately chosen.
\begin{proposition}
Let $\beta > 0$. Then, problem \eqref{Prob_Pen} is equivalent to \eqref{Prob_Ori} if and only if $\beta > d_p$. Specifically, any global minimizer $X$ of \eqref{Prob_Pen} has the form 
\begin{equation*}
X := \bar{S}\left(I_{2p} - \frac{\bar{D}}{\beta}\right)^{1/2}T^\top,
\end{equation*}
where $\bar{S} \in \text{Sp}(2p, 2n)$ satisfies
\begin{equation*}
\bar{S}^\top A \bar{S} = \bar{D} = \mathrm{Diag}\ssp{d_1, \dots, d_p, d_1, \dots, d_p},
\end{equation*}
and $T \in \text{OrSp}(2p)$.
\end{proposition}
\begin{proof}
By Proposition~\ref{prop:stpform}, consider a global minimizer $X$ of \eqref{Prob_Pen} expressed as:
\begin{equation*}
X := \begin{bmatrix}
\hat{S}_1\left(I_q - \frac{\hat{D}}{\beta}\right)^{1/2} & 0_{2n \times (p-q)} & \hat{S}_2\left(I_q - \frac{\hat{D}}{\beta}\right)^{1/2} & 0_{2n \times (p-q)}
\end{bmatrix} T^\top.
\end{equation*}
The corresponding objective value is $f_\beta(X) = \frac{1}{2} \text{tr}(\hat{D}) - \frac{1}{4\beta} \text{tr}(\hat{D}^2) + \beta (p-q)$. Then let $Y := \bar{S}\left(I_{2p} - \frac{\bar{D}}{\beta}\right)^{1/2}T^\top$, we have $f_\beta(Y) = \frac{1}{2} \text{tr}(\bar{D}) - \frac{1}{4\beta} \text{tr}(\bar{D}^2)$. 

If $\beta > d_p$, the fact $f_\beta(Y) \geq f_\beta(X)$ implies $\hat{D} = \bar{D}$. Conversely, if $\beta < d_p$, all global minimizers are rank-deficient. This completes the proof. 
\end{proof}

A remarkable property of the penalty formulation is the absence of local minimizers, ensuring that any local optimization method converging to a stationary point will likely find the global solution.

\begin{theorem}\label{thm:nolocal}
For any $\beta > d_p$, the penalty problem \eqref{Prob_Pen} has no local minimizers. Specifically, each first-order stationary point of \eqref{Prob_Pen} is either a global minimizer, a saddle point, or a local maximizer.
\end{theorem}

\begin{proof}
The directional Hessian at any stationary point $X \in \mathbb{R}^{2n \times 2p}$ is given by:
\begin{equation}
\label{eq:hess}
\begin{aligned}
\nabla^2 f_\beta(X)(Y, Y) = & \text{tr}\left(Y^\top A Y\right) - \beta \, \text{tr}\left(\left(Y^\top J_n Y\right)\left(X^\top J_n X - J_p\right)\right) \\
& + \frac{\beta}{2}\left\|Y^\top J_n X + X^\top J_n Y\right\|_F^2.
\end{aligned}
\end{equation}

By Proposition~\ref{prop:stpform}, assume $X$ is not a global minimizer and has the form
\begin{equation*}
X := \begin{bmatrix}
\hat{S}_1\left(I_q - \frac{\hat{D}}{\beta}\right)^{1/2} & 0_{2n \times (p-q)} & \hat{S}_2\left(I_q - \frac{\hat{D}}{\beta}\right)^{1/2} & 0_{2n \times (p-q)}
\end{bmatrix} T^\top,
\end{equation*}
where $\hat{S} := \begin{bmatrix} \hat{S}_1 & \hat{S}_2 \end{bmatrix} \in \text{Sp}(2q, 2n)$ diagonalizes $A$ as $\hat{S}^\top A \hat{S} = \hat{D}\oplus\hat{D}$.

Let $\bar{S} \in \text{Sp}(2p, 2n)$ be such that $\bar{S}^\top A \bar{S} = \bar{D}$, where
\begin{equation*}
\bar{D} := \text{Diag}(d_1, \dots, d_p, d_1, \dots, d_p).
\end{equation*}
Then, the directional Hessian along $\bar{S}$ satisfies
\begin{equation*}
\nabla^2 f_\beta(X)\left(\bar{S}, \bar{S}\right) = \text{tr}(\bar{D}) - 2\text{tr}(\hat{D}) - \beta (p-q) < 0.
\end{equation*}

This implies that $X$ is either a saddle point or a local maximizer. Since no local minimizers exist when $\beta > d_p$, the proof is complete.
\end{proof}

Theorem \ref{thm:nolocal} directly implies the following Corollary \ref{coro:sec-glo}, which guarantees that the process of finding a second-order necessary stationary point also yields a global minimizer.

\begin{corollary}\label{coro:sec-glo}
Suppose $\beta > d_p$ and $X$ is a second-order necessary stationary point of \eqref{Prob_Pen}, then  $X$ is a global minimizer of \eqref{Prob_Ori}. 
\end{corollary}

\section{Algorithmic Framework}
\label{sec:alg}

\subsection{Gradient Methods for the Penalty Problem}
Problem \eqref{Prob_Pen} is an unconstrained nonconvex minimization problem, solvable by established methods such as steepest descent, conjugate gradient, Newton, and quasi-Newton methods. Due to the computational cost of eigenvalue computations, we focus on gradient-type methods of the form 
\begin{equation}
\label{eq:iter}
X^{(k+1)} := X^{(k)} - \gamma^{(k)}\nabla f_\beta\left(X^{(k)}\right),
\end{equation}
where $\gamma^{(k)}$ is the step-size at iteration $k$.

The equivalence between \eqref{Prob_Pen} and \eqref{Prob_Ori} ensures that a variety of unconstrained optimization methods can solve \eqref{Prob_Ori} by addressing \eqref{Prob_Pen}. Examples include the gradient descent method, conjugate gradient method, Newton's method, and quasi-Newton methods. Considering the computational cost of solving the symplectic eigenvalue problem, this section focuses on gradient-type methods. Additionally, under random initialization, gradient-type methods can almost surely find second-order stationary points of unconstrained optimization problems. Thus, by Corollary \ref{coro:sec-glo}, this guarantees that we obtain a global minimizer of \eqref{Prob_Pen}.

In this section, we consider using the gradient method with BB step-size to solve \eqref{Prob_Pen} and prove the corresponding convergence properties.

While $f_\beta$ may have multiple stationary points, Theorem \ref{thm:nolocal} shows that all stationary points except global minimizers are saddle points or maximizers, regardless of rank. This differs from penalty methods for standard eigenvalue problems \cite{Wen2016TracePenalty}, where preserving full rank of iterates is critical. Nevertheless, we derive the following result by analogy to \cite{Wen2016TracePenalty}.

\begin{proposition}\label{prop:rank-deficient}
Let $X^{(k+1)}$ be generated by Algorithm \ref{alg:BB} from a full-rank iterate $X^{(k)}$. Then $X^{(k+1)}$ is rank-deficient if and only if $1/\gamma^{(k)}$ is an eigenvalue of the generalized symmetric eigenvalue problem
\begin{equation*}
\left({X^{(k)}}\tp\nabla f_\beta\left(X^{(k)}\right)\right)u = \lambda \left({X^{(k)}}\tp X^{(k)}\right)u.
\end{equation*}
In particular, if $\gamma^{(k)}<\sigma_{\text{min}}\left(X^{(k)}\right)/\|\nabla f_\beta\left(X^{(k)}\right)\|_2$, then $X^{(k+1)}$ remains full-rank.
\end{proposition}

\begin{proof}
Assume $X^{(k)}$ is full-rank and $X^{(k+1)}$ is rank-deficient. By definition, there exists a nonzero $u$ such that
\begin{equation*}
0 = {X^{(k)}}\tp X^{(k+1)}u = {X^{(k)}}\tp X^{(k)}u - \gamma^{(k)}{X^{(k)}}\tp\nabla f_\beta\left(X^{(k)}\right)u.
\end{equation*}
Rearranging confirms $1/\gamma^{(k)}$ satisfies the generalized eigenvalue problem, as \linebreak ${X^{(k)}}\tp\nabla f_\beta(X^{(k)})$ is symmetric (by \eqref{eq:grad-fbeta-X-0}).
\end{proof}

Algorithm \ref{alg:BB} presents a non-monotonic gradient method for \eqref{Prob_Pen}, using the Barzilai-Borwein (BB) step-size and a Grippo-Lucidi (GLL)-type line search to guarantee global convergence. Specifically, set $S^{(k-1)}:=X^{(k)}-X^{(k-1)}$ and $Z^{(k-1)}:=\nabla f_\beta\ssp{X^{(k)}}-\nabla f_\beta\ssp{X^{(k-1)}}$, then the alternative BB step-size is
\begin{equation*}
\gamma^{(k)}_\text{BB} :=
\begin{cases}
\frac{\inner{S^{(k-1)},S^{(k-1)}}}{\left|\inner{S^{(k-1)},Z^{(k-1)}}\right|}\quad \text{for even }k\\
\frac{\left|\inner{S^{(k-1)},Z^{(k-1)}}\right|}{\inner{Z^{(k-1)},Z^{(k-1)}}}\quad \text{for odd }k
\end{cases}
,\quad m\in\bb{N}.
\end{equation*}

\begin{algorithm}[t]
\caption{Symplectic Eigenspace by Penalty – Basic Version (SympEigPen-B)}\label{alg:BB}
\begin{algorithmic}[1]
\Require{Symmetric positive definite matrix $A\in\mathbb{R}^{2n\times 2n}$, initial iterate $X^{(0)}\in\mathbb{R}^{2n\times 2p}$, penalty parameter $\beta>0$, step size bounds $\bar{\gamma}\geq\gamma^{(0)}\geq\underline{\gamma}>0$, maximum iterations $k_{\text{max}}\in\mathbb{N}$, tolerance $\epsilon>0$, line search parameters $0<\delta,\lambda <1$, and memory length $L\in\mathbb{N}$.}
\Ensure{Approximate solution $X^{(k)}$.}
\For{$k=0,1,\dots,k_{\text{max}}$}
\State Compute $G^{(k)}:=\nabla f_\beta\left(X^{(k)}\right)$.
\If{$\left\|G^{(k)}\right\|_F<\epsilon$}
\State $\textbf{break}$
\EndIf
\If{$k>0$}
\State Compute $\gamma^{(k)}:=\max\left\{\underline{\gamma},\min\left\{\bar{\gamma},\gamma^{(k)}_\text{BB}\right\}\right\}$.
\EndIf
\State Find the smallest integer $t\geq 0$ satisfying
\begin{equation}\label{eq:GLL}
f_\beta\left(X^{(k)}-\delta^t\gamma^{(k)}G^{(k)}\right)\leq \max\limits_{0\leq l\leq \min(k,L)}f_\beta\left(X^{(k-l)}\right) - \lambda\cdot\delta^t\gamma^{(k)}\|G^{(k)}\|_F^2
\end{equation}
\State Update $X^{(k+1)}:=X^{(k)}-\delta^t\gamma^{(k)}G^{(k)}$
\EndFor
\end{algorithmic}
\end{algorithm}

The global convergence of Algorithm \ref{alg:BB} is established by the following theorem, addressing the limiting case without termination.

\begin{theorem}[Global convergence of Algorithm \ref{alg:BB}]
\label{thm:BB-convergence}
Let $G^{(k)}$ be generated by Algorithm \ref{alg:BB}, then $\lim\limits_{k\rightarrow\infty} \|G^{(k)}\|_F = 0$.
\end{theorem}

\begin{proof}
There exists $C>0$ (dependent on $\beta$) such that $\|\nabla^2 f_\beta\|_F \leq C$. Using this Lipschitz bound on the Hessian we yield
\begin{equation*}
\begin{aligned}
f_\beta\left(X^{(k)}-\delta^t\gamma^{(k)}G^{(k)}\right) &\leq f_\beta\left(X^{(k)}\right) - \delta^t\gamma^{(k)}\|G^{(k)}\|_F^2 + C\delta^{2t}{\gamma^{(k)}}^2\|G^{(k)}\|_F^2 \label{leq:descent}\\
&\leq \max\limits_{0\leq l\leq L}f_\beta\left(X^{(k-l)}\right) - \left(\delta^t\gamma^{(k)}-C\delta^{2t}{\gamma^{(k)}}^2\right)\|G^{(k)}\|_F^2\\
&= \max\limits_{0\leq l\leq L}f_\beta\left(X^{(k-l)}\right) - \delta^t\gamma^{(k)}\left(1-C\delta^{t}\gamma^{(k)}\right)\|G^{(k)}\|_F^2.
\end{aligned}
\end{equation*}
A sufficient condition for \eqref{eq:GLL} is $1-C\delta^{t}\gamma^{(k)}\geq\lambda$, or $t\leq\log_\delta\frac{1-\lambda}{C\gamma^{(k)}}$.

Since $\bar{\gamma}\geq\gamma^{(k)}\geq\underline{\gamma}>0$, the line search parameter $t^{(k)}$ at iteration $k$ has a uniform upper bound:
\begin{equation}\label{eq:tk}
t^{(k)}\leq \log_\delta\frac{1-\lambda}{C\bar{\gamma}} := \bar{t}.
\end{equation}
As $f_\beta$ is bounded below, 

By \eqref{leq:descent}, we obtain
\begin{equation*}
\begin{aligned}
f_\beta\ssp{X^{(k)}} - f_\beta\ssp{X^{(k+1)}} &\geq \delta^{t^{(k)}}\gamma^{(k)}\left(1-C\delta^{t^{(k)}}\gamma^{(k)}\right)\|G^{(k)}\|_F^2\\
&\geq \delta^{\bar{t}}\underline{\gamma}\lambda\|G^{(k)}\|_F^2.
\end{aligned}
\end{equation*}
Let $k=0,1,\dots$ and summing up, we obtain
\begin{equation*}
\sum_{k=0}^\infty\delta^{\bar{t}}\underline{\gamma}\lambda\|G^{(k)}\|_F^2 < +\infty,
\end{equation*}
hence $\lim\limits_{k\rightarrow\infty} \|G^{(k)}\|_F = 0$.
\end{proof}

\subsection{Enhancement by Restarting and Randomized Step-Size}
Algorithm \ref{alg:BB} exhibits typical gradient method behavior, beginning with rapid objective reduction initially, followed by slowdown near minimizers. Restarting with a modified iterate accelerates convergence, as shown for \eqref{Prob_Pen}. We introduce a restarting strategy with multiple symplectic Rayleigh–Ritz (SRR) steps, analogous to \cite{Wen2016TracePenalty} (which uses conventional Rayleigh–Ritz steps).

Proposition \ref{prop:rank-deficient} demonstrates that rank deficiency arises only for specific discrete $\gamma^{(k)}$. Introducing randomness into step-size updates makes rank deficiency occur with probability zero.
Therefore, integrating these enhancements yields Algorithm \ref{alg:BB-enh}, an improved version of Algorithm \ref{alg:BB}.
Convergence of Algorithm \ref{alg:BB-enh} follows from Theorem \ref{thm:BB-convergence}, as $\lim\limits_{i \rightarrow \infty} \varepsilon_i = 0$, formalized below.

\begin{algorithm}[t]
\caption{Symplectic Eigenspace by Penalty – Enhanced Version (SympEigPen)}\label{alg:BB-enh}
\begin{algorithmic}[1]
\Require{Symmetric positive definite matrix $A \in \mathbb{R}^{2n \times 2n}$.}
\Ensure{Approximate solution $\bar{X}^{(i)}$.}
\State Initialize $\bar{X}^{(0)} \in \mathbb{R}^{2n \times 2p}$, $\beta > 0$, initial step size and bounds $\bar{\gamma} \geq \gamma^{(0)} \geq \underline{\gamma} > 0$, randomization bounds $\bar{\xi} > \underline{\xi} > 0$, maximum inner iterations $k_{\text{max}} \in \mathbb{N}$, tolerance reduction factor $0 < \delta_\varepsilon < 1$, initial tolerance $\epsilon_0 > 0$, line search parameters $0 < \delta, \lambda < 1$, memory length $L \in \mathbb{N}$, factor $\eta > 0$ for updating $\beta$ and set $i = k := 0$.
\While{convergence criterion not met}
\State Set $X^{(k)} := \bar{X}^{(i)}$.
\For{$k = 0, 1, \dots, k_{\text{max}}$}
\State Compute $G^{(k)} := \nabla f_\beta\left(X^{(k)}\right)$.
\If{$\|G^{(k)}\|_F < \varepsilon_i \cdot \max\left\{1, \|A X^{(k)}\|_F\right\}$}
\State \textbf{break}
\EndIf
\If{$k > 0$}
\State Draw $\xi^{(k)} \sim \mathcal{U}[\underline{\xi}, \bar{\xi}]$, and compute $\gamma^{(k)} := \xi^{(k)} \cdot \max\left\{\underline{\gamma}, \min\left\{\bar{\gamma}, \gamma^{(k)}_\text{BB}\right\}\right\}$.
\EndIf
\State Find the smallest integer $t \geq 0$ satisfying the line search condition
\begin{equation*}
f_\beta\left(X^{(k)} - \delta^t \gamma^{(k)} G^{(k)}\right) \leq \max_{0 \leq l \leq L} f_\beta\left(X^{(k-l)}\right) - \lambda \cdot \delta^t \gamma^{(k)} \|G^{(k)}\|_F^2
\end{equation*}
\State Update $X^{(k+1)} := X^{(k)} - \delta^t \gamma^{(k)} G^{(k)}$.
\EndFor
\State Update $\beta := \eta\cdot\max\limits_iD_{ii}$.
\State Execute the SRR procedure $(S, D) = \mathrm{SRR}\left(X^{(k)}\right)$ and set $\bar{X}^{(i+1)} := S \left(1 - D / \beta\right)^{1/2}$.
\State Update tolerance $\varepsilon_{i+1} := \delta_\epsilon \varepsilon_i$ and set $i := i + 1$.
\EndWhile
\end{algorithmic}
\end{algorithm}

\begin{corollary}
Let $\bar{X}^{(i)}$ be generated by Algorithm \ref{alg:BB-enh}, then $\lim\limits_{i \rightarrow \infty} \left\|\nabla f_\beta\left(\bar{X}^{(i)}\right)\right\|_F = 0$.
\end{corollary}

The following proposition demonstrates that randomized step-size strategy guarantees rank preservation in the sense of probability.

\begin{proposition}
Let $\bar{X}^{(i)}$ be generated by Algorithm \ref{alg:BB-enh}. If the algorithm terminates in finitely many iterations, then $\left\{\bar{X}^{(i)}\right\}$ remains full-rank with probability one.
\end{proposition}

\begin{proof}
By Proposition \ref{prop:rank-deficient}, rank deficiency requires $\gamma^{(k)}$ to equal one of finitely many values (reciprocals of eigenvalues from a generalized eigenvalue problem). Randomization ensures the probability of selecting such $\gamma^{(k)}$ is zero at each iteration. Thus,
\[
P_i := \mathbb{P}\left(\bar{X}^{(i+1)} \text{ is full-rank} \mid \bar{X}^{(j)} \text{ is full-rank for all } j \leq i\right) = 1.
\]
Let $i_{\text{max}}$ be the termination iteration. The joint probability of all $\bar{X}^{(j)}$ ($j \leq i_{\text{max}}$) being full-rank is
\[
\mathbb{P}\left(\bar{X}^{(j)} \text{ is full-rank for all } j \leq i_{\text{max}}\right) = \prod_{j=0}^{i_{\text{max}}-1} P_j = 1,
\]
completing the proof.
\end{proof}

\subsection{Penalty Parameter Adjustment}

In penalty function methods, the performance of numerical optimizers is highly sensitive to the choice of the penalty parameter, for which a suitable value is often difficult to determine a priori. When $p > 1$, the global minimizers of $f_\beta$ are non-isolated. To characterize the local convergence rate, we estimate the condition number of the projected Hessian at the global minimizers of \eqref{Prob_Pen} and show that it is independent of the penalty parameter $\beta$. We then examine a specific directional Hessian in the subspace spanned by the global minimizer, whose value depends on $\beta$. Based on this analysis, we derive a theoretical optimum for $\beta$ and propose a practical estimation of this value. Finally, since the proposed \textit{SympEigPen} algorithm incorporates a restart mechanism, we introduce a self-adaptive strategy for adjusting $\beta$ during each restart.

\begin{proposition}
Let $p > 1$, $\beta > d_p$, and let $X$ be a global minimizer of $f_\beta$ in \eqref{Prob_Pen}. The condition number of the Hessian restricted to $\mathrm{span}\ssp{\hat{S}}$ satisfies
\begin{equation*}
\kappa\left(\nabla^2 f_\beta(X)\mid_{\mathrm{span}\ssp{\hat{S}}}\right) \leq \frac{d_n+d_p}{d_{p+1}-d_p} \cdot \frac{d_n}{d_{p+1}}\cdot\kappa(A).
\end{equation*}
where $\hat{S} \in \mathbb{R}^{2n \times (2n-2p)}$ contains symplectic eigenvector pairs for eigenvalues $d_{p+1}, \dots, d_n$, and $\kappa(A)$ is the spectral condition number of $A$.
\end{proposition}

\begin{proof}
Assume $X:=\bar{S}\left(I_{2p}-\bar{D}/\beta\right)^{1/2}T\tp$ is a global minimizer of \eqref{Prob_Pen}. For arbitrary $Y\in\mathrm{span}\ssp{\hat{S}}$ with $\tr\ssp{Y\tp Y}=1$, let $Y=\hat{S}\hat{Q}$.
From \eqref{eq:hess} we yield
\begin{equation*}
\begin{aligned}
\nabla^2 f_\beta(X)(Y, Y) = & \tr\left(Y\tp AY\right) - \beta \tr\left(\left(Y\tp J_nY\right)\left(X\tp J_n X-J_p\right)\right) \\
& + \frac{\beta}{2}\left\|Y\tp J_nX+X\tp J_nY\right\|_F^2,
\end{aligned}
\end{equation*}
implying
\begin{equation*}
\begin{aligned}
\mathrm{tr}\left(Y\tp AY\right) = \tr\left(\hat{Q}\tp\hat{D}\hat{Q}\right)\geq d_{p+1}\cdot\tr\left(\hat{Q}\hat{Q}\tp\right).
\end{aligned}
\end{equation*}
and
\begin{equation*}
\begin{aligned}
\mathrm{tr}\left(Y\tp AY\right) \leq d_{n}\cdot\tr\left(\hat{Q}\hat{Q}\tp\right).
\end{aligned}
\end{equation*}
Moreover, the second term can be bounded as follows
\begin{equation*}
\begin{aligned}
\left|\beta \tr\left(\left(Y\tp J_nY\right)\left(X\tp J_n X-J_p\right)\right)\right|=\left|\mathrm{tr}\left(\left(Y\tp J_nYJ_p\right)\left(T\bar{D}T\tp\right)\right)\right| \leq d_{p}\cdot\tr\left(\hat{Q}\hat{Q}\tp\right).
\end{aligned}
\end{equation*}
The third term vanishes when $X$ is a global minimizer.
Since $\mathrm{tr}\left(Y\tp Y\right)=1$, we obtain that $1=\mathrm{tr}\left(\hat{Q}\tp S\tp S\hat{Q}\right)\geq\mathrm{tr}\left(\hat{Q}\hat{Q}\tp\right)\cdot\sigma_{\mathrm{min}}\left(S\tp S\right)$, and $1=\mathrm{tr}\left(Q\tp S\tp SQ\right)\leq\mathrm{tr}\left(\hat{Q}\hat{Q}\tp\right)\cdot\sigma_{\mathrm{max}}\left(S\tp S\right)$. Therefore,
\begin{equation*}
\begin{aligned}
\frac{d_{p+1}-d_p}{\sigma_{\text{max}}\left(\hat{S}\tp\hat{S}\right)} \leq \mathrm{tr}\left(Y\tp\nabla^2 f_\beta(X)(Y)\right) \leq \frac{d_n + d_p}{\sigma_{\text{min}}\left(\hat{S}\tp\hat{S}\right)}.
\end{aligned}
\end{equation*}
The condition number can be estimated by
\begin{equation*}
\kappa\left(\nabla^2 f_\beta(X)\mid_{\mathrm{span}\ssp{\hat{S}}}\right) \leq \frac{d_n+d_p}{d_{p+1}-d_p} \cdot \frac{\sigma_{\mathrm{max}}\left(\hat{S}\tp\hat{S}\right)}{\sigma_{\mathrm{min}}\left(\hat{S}\tp\hat{S}\right)}.
\end{equation*}
Since $\hat{S}\tp A\hat{S} = \hat{D}$, we obtain
\begin{equation*}
\frac{\sigma_{\mathrm{max}}\left(\hat{S}\tp\hat{S}\right)}{\sigma_{\mathrm{min}}\left(\hat{S}\tp\hat{S}\right)} \leq \frac{d_n}{d_{p+1}}\cdot\kappa(A).
\end{equation*}
Therefore, we conclude that
\begin{equation*}
\kappa\left(\nabla^2 f_\beta(X)\mid_{\mathrm{span}\ssp{\hat{S}}}\right) \leq \frac{d_n+d_p}{d_{p+1}-d_p} \cdot \frac{d_n}{d_{p+1}}\cdot\kappa(A).
\end{equation*}
\end{proof}

This result shows the condition number bound within $\mathrm{span}\ssp{\hat{S}}$ is independent of $\beta$. However, the convergence performance of algorithms also depends on curvature in $\mathrm{span}\ssp{\bar{S}}$. The following guides $\beta$ selection via a specific direction in this subspace.

\begin{proposition}\label{prop:beta}
Let $p > 1$, $\beta > d_p$, and let $X$ be a global minimizer of $f_\beta$ in \eqref{Prob_Pen}. There exists a specific direction $Y$ with $\tr\ssp{Y\tp Y} = 1$ satisfying
\begin{equation*}
\frac{\sigma_\mathrm{min}(A)}{d_p} \cdot \beta\ssp{1-\frac{d_p}{\beta}}^2\leq \nabla^2 f_\beta(X)(Y, Y) \leq \frac{\sigma_\mathrm{max}(A)}{d_p} \cdot \beta\ssp{1-\frac{d_p}{\beta}}^2.
\end{equation*}
\end{proposition}

\begin{proof}
Let $X := \bar{S} (I_{2p} - \bar{D} / \beta)^{1/2} T^{\top}$ and $Z$ consists of all zero columns except two columns of symplectic eigenvector pair corresponding to $d_p$.
Substituting $ZT\tp$ into \eqref{eq:hess} gives
\begin{equation*}
\begin{aligned}
\nabla^2 f_\beta(X)\ssp{ZT\tp, ZT\tp} &= 2d_p - 2d_p + \frac{\beta}{2} \cdot 2\ssp{1-\frac{d_p}{\beta}}^2 = \beta\ssp{1-\frac{d_p}{\beta}}^2.
\end{aligned}
\end{equation*}
Moreover, $\frac{d_p}{\sigma_{\mathrm{max}}(A)}\leq \|ZT\tp\|_F^2 \leq \frac{d_p}{\sigma_{\mathrm{min}}(A)}$. Let $Y:= ZT\tp/\|ZT\tp\|_F$, then the bound is obtained.
\end{proof}

It is evident from Proposition \ref{prop:beta} that $\beta$ cannot be too near $d_p$ because the bounds tend to $0$, which results in ill-conditioning. As the optimal $\beta$ should be at the same scale as the directional Hessians in $\mathrm{span}\ssp{\hat{S}}$, $\beta\ssp{1-\frac{d_p}{\beta}}^2$ should fall between the scales of $d_{p+1}-d_p$ and $d_n + d_p$.
Let 
\begin{equation*}
\beta\ssp{1-\frac{d_p}{\beta}}^2 = d_p
\end{equation*}
and we yield $\beta_{\mathrm{best}} := \frac{3+\sqrt{5}}{2}d_p$. Since $\tr(A) \geq 2(n-p+1) d_p$ (trace minimization theorem), a practical heuristic is $\beta_{\mathrm{sug}} := \frac{\tr(A)}{n-p+1}$. For our proposed Algorithm, we can refine $\beta$ to $\beta := \eta \theta_p$ ($\eta > 1$) after an SRR step, where $\theta_p$ refers to the $p$-th Ritz value.



\subsection{Cost Analysis and Comparison}

In this subsection, we analyze the computational costs of our proposed \textit{SympEigPen} method and Riemannian optimization approaches for solving the symplectic eigenvalue problem. The computational complexities of these algorithms are governed by different factors, leading to distinct performance profiles. We provide a concise explanation for the superior efficiency of our method.

The cost of Riemannian optimization methods is primarily dominated by a few products between the $2n \times 2n$ matrix $A$ and $2n \times 2p$ matrices, each with complexity $\mathcal{O}\ssp{\text{nnz} \cdot p}$, where nnz denotes the number of nonzero elements in $A$. While this is acceptable when $A$ is sparse, other operations, such as solving Lyapunov equations and performing retractions, incur an additional cost of $\mathcal{O}(np^2 + p^3)$. Although $p < n$ typically holds, the constant factor corresponding to the $p^3$ term is usually large. Furthermore, each iteration involves multiple products between $2n \times 2p$ and $2p \times 2p$ matrices, which collectively contribute significantly to the overall cost.

In contrast, the proposed formulation in \eqref{Prob_Pen} achieves a per-iteration cost of $\mathcal{O}(np)$ when using first-order optimization methods, excluding the cost of evaluating the function and its gradient. These evaluations generally require $\mathcal{O}(\text{nnz} \cdot p + np^2)$, where the $np^2$ term arises only twice from computing $X^\top J_n X$ and \linebreak $J_n X \ssp{X^\top J_n X - J_p}$. Moreover, \textit{SympEigPen} requires far fewer products between $2n \times 2p$ and $2p \times 2p$ matrices per iteration than Riemannian approaches. Although the restart procedure also has a cost of $\mathcal{O}(np^2 + p^3)$, similar to that in Riemannian methods, it is invoked infrequently, thus having minimal impact on the overall performance.

The computational costs of the basic linear algebra operations, as well as of evaluating the function value and gradient, are provided in Table \ref{tab:cost}. A comparison of key operations between Riemannian optimization and TPM-based approaches is presented in Table \ref{tab:cost-comparison}.

\begin{table}[htbp!]
\centering
\begin{tabular}{|c|c|c|}
\hline
\multirow{3}{*}{Computing basic quantities} & $AX$ & $\text{nnz}\cdot 2p$ \\
 & $X\tp J_n X$ & $8np^2$ \\
 & $X\tp J_n X - J_p$ & $4p^2$ \\
\hline
\multirow{3}{*}{Computing $f_\beta(X)$} & $\tr\ssp{X\tp AX}$ & $8np$ \\
 & $\|X\tp J_n X - J_p\|_F^2$ & $12p^2$ \\
 & $\frac{1}{2}\tr\ssp{X\tp AX} + \frac{\beta}{4}\|X\tp J_n X - J_p\|_F^2$ & $1$ \\
\hline
\multirow{2}{*}{Computing $\nabla f_\beta(X)$} & $J_n X \ssp{X\tp J_n X - J_p}$ & $8np^2$ \\
 & $AX -\beta J_n X \ssp{X\tp J_n X - J_p}$ & $4np$ \\
\hline
In total & \multicolumn{2}{c|}{$\text{nnz}\cdot 2p + 16np^2$}\\ \hline
\end{tabular}
\caption{Computational complexity of the zero- and first-order oracles in \eqref{Prob_Pen}.}
\label{tab:cost}
\end{table}

\begin{table}[htbp!]
\centering
\resizebox{\linewidth}{!}{
\begin{tabular}{|cc|cc|}
\hline
\multicolumn{2}{|c|}{Riemannian optimization approaches} & \multicolumn{2}{c|}{TPM based approaches} \\
\hline
\multirow{3}{*}{Riemannian gradient} & Solving Lyapunov equation: $> 200p^3$ & \multirow{3}{*}{Euclidean gradient} & \multirow{3}{*}{$\nabla f_\beta(X)$ : $\text{nnz}\cdot 2p + 16np^2$} \\
 & Computing inverse $\ssp{X\tp X}^{-1}$: $> 8np^2 + 16p^3$ & & \\
 & In total: $> \text{nnz}\cdot 2p + 72np^2 + 16p^3$ & & \\
\hline
\multirow{2}{*}{Retraction} & Cayley retraction: $> 40np^2 + 16p^3$ & \multirow{2}{*}{No retraction} & -- \\
 & Quasi-geodesic retraction: $> 32np^2 + 1280p^3$ & & -- \\
\hline
\end{tabular}
}
\caption{Comparison on the computational complexity of the first-order oracles among Riemannian optimization approaches and TPM based approaches.}
\label{tab:cost-comparison}
\end{table}

\section{Numerical Experiments}
\label{sec:numexp}
In this section, we evaluate the numerical performance of applying unconstrained optimization approaches for solving \eqref{Prob_Ori} through \eqref{Prob_Pen}.  All the numerical experiments were performed on a workstation with two Intel(R) Xeon(R) Gold 5317 processors (3.00 GHz × 12, 18M Cache), 512 GB RAM, Python 3.9, and Ubuntu 20.04.1. All Python packages, including NumPy 2.3.4 and SciPy 1.16.2, are installed via the Conda package manager, with a primary reliance on the conda-forge channel.

The equivalence between \eqref{Prob_Ori} and \eqref{Prob_Pen} yields direct implementations of existing unconstrained optimization approaches for solving \eqref{Prob_Ori}. As a result, we tested \texttt{L-BFGS-B}, \texttt{TNC} (both from SciPy), and our BB method in Algorithm \ref{alg:BB}. Comparisons include \texttt{eigs} (SciPy, applied to $J_nA$), and the Riemannian optimization method from \cite{Son2021Computing} (adapted from MATLAB \cite{Son2021Computing} to Python). The Riemannian optimization uses identical parameters from \cite{Son2021Computing}. For \textit{TPM}, $\beta$ was tuned and iterations stopped when $\|\nabla f_\beta\|_F < 10^{-7}$. The other parameters are set as $\bar{X}^{(0)} := \mmp{I_{2p}[\cdot,1 : p], I_{2p}[\cdot,p + 1 : 2p]}$, $\gamma^{(0)} := 10^{-4}, \bar{\gamma} :=10^5, \underline{\gamma} := 10^{-8}$, $\bar{\xi} = 1, \underline{\xi} := 0.99$, $k_{\text{max}} := 5000$, $\delta_\varepsilon := 0.1$, $\epsilon_0 := 0.1$, $\delta:=0.5, \lambda := 10^{-8}$, $L := 50$ and $\eta:=1.1$.

Three matrix types are used in the experiments. A \textit{Dense} matrix is defined by $A := \frac{1}{2}\left(B+B\tp\right)$ with random $B\in\mathbb{R}^{2n\times 2n}$, $B_{ij}\sim\mathcal{N}(0,1)$ and scaled to have $\sigma_{\text{min}}(A)=1$, $\sigma_{\text{max}}(A)=n$ by adding $\lambda I_{2n}$. With the same scaling, a \textit{Sparse} matrix is randomly generated by \texttt{scipy.sparse.random} with $\sigma_{\text{min}}(A)=1$, $\sigma_{\text{max}}(A)=n$, density $\sigma = 10/n$. A \textit{Sparse-add-low-rank} matrix is defined by $A := B + CC\tp$ with sparse $B$ (density $\sigma=10/n$, scaled to $\sigma_{\text{min}}(B)=1$, $\sigma_{\text{max}}(B)=n$) and random $C\in\mathbb{R}^{2n\times m}$ ($C_{ij}\sim\mathcal{U}[-1,1]$, scaled to $\sigma_{\text{max}}(CC\tp)=n$, $m=10$).

For a solution $X$ versus reference $X_{\text{eigs}}$ (symplectic eigenvectors computed by \texttt{eigs}), the Golub-Werman subspace error is defined as
\begin{equation*}
\mathcal{E}:=\left\|X(X\tp X)^{-1}X\tp - X_\mathrm{eigs}(X_\mathrm{eigs}\tp X_\mathrm{eigs})^{-1}X_\mathrm{eigs}\tp\right\|_F.
\end{equation*}
For large $n$ (when \texttt{eigs} is inaccurate or slow), we use the residue instead, which is defined as
\begin{equation*}
\mathcal{R}:=\frac{\|AX - J_nXJ_p\tp D\|_F}{\|AX\|_F}.
\end{equation*}

\subsection{Penalty Parameter Tuning}
We conduct numerical experiments to tune the penalty parameter $\beta$ for problem \eqref{Prob_Pen} with $n=200$ and $p=10$, testing values $\beta := 1.001d_p$, $\beta_{\text{best}}$, $\beta_{\text{sug}}$, $2\beta_{\text{sug}}$, $5\beta_{\text{sug}}$, $10\beta_{\text{sug}}$, and $100\beta_{\text{sug}}$. The iteration counts for three different matrix types—dense (Table \ref{tab:beta-dense}), sparse (Table \ref{tab:beta-sparse}), and sparse-plus-low-rank (Table \ref{tab:beta-sparse-add-low-rank})—are reported. Several key observations emerge from the results.

First, when $\beta$ is chosen very close to the theoretical lower bound $d_p$ (for guarantee of the exact penalization property), the convergence slows significantly across all algorithms and matrix types. This behavior is consistent with the theoretical prediction in Proposition \ref{prop:beta}, which establishes that $\beta$ should be slightly larger than $d_p$ to avoid ill-conditioning.

Second, the suggested value $\beta_{\text{sug}}$ demonstrates remarkable robustness, performing nearly optimally across all tested scenarios. In dense matrices (Table \ref{tab:beta-dense}), $\beta_{\text{sug}}$ achieves iteration counts very close to the empirically optimal $\beta_{\text{best}}$. Similar patterns hold for sparse and sparse-plus-low-rank matrices, where $\beta_{\text{sug}}$ maintains efficient convergence while avoiding the poor performance at extreme $\beta$ values.

Third, excessively large values lead to substantially increased iteration counts, approaching the poor performance observed at $\beta \approx d_p$. This U-shaped performance pattern underscores the importance of selecting $\beta$ within an appropriate range—neither too close to $d_p$ nor excessively large.

\begin{table}[htbp!]
\centering
\begin{tabular}{|c|ccccccc|}
\hline
$\beta$ &$1.001d_p$&$\beta_\mathrm{best}$ &$\beta_\mathrm{sug}$ &$2\beta_\mathrm{sug}$ &$5\beta_\mathrm{sug}$ &$10\beta_\mathrm{sug}$ &$100\beta_\mathrm{sug}$\\
\hline
LBFGS&1318&350&377&378&387&437&1208\\ 
TNC&37&29&31&35&38&41&112\\ 
SympEigPen-B&9326&565&611&638&819&990&2474\\ \hline
\end{tabular}
\caption{Iteration counts for dense $A$ and varying $\beta$.}
\label{tab:beta-dense}
\end{table}

\begin{table}[htbp!]
\centering
\begin{tabular}{|c|ccccccc|}
\hline
$\beta$ &$1.001d_p$&$\beta_\mathrm{best}$&$\beta_\mathrm{sug}$ &$2\beta_\mathrm{sug}$ &$5\beta_\mathrm{sug}$ &$10\beta_\mathrm{sug}$ &$100\beta_\mathrm{sug}$\\
\hline
LBFGS&892&640&650&683&698&770&1281\\ 
TNC&40&35&36&43&55&59&105\\ 
SympEigPen-B&6996&988&1077&1604&1819&1904&3872\\ \hline
\end{tabular}
\caption{Iteration counts for sparse $A$ and varying $\beta$.}
\label{tab:beta-sparse}
\end{table}

\begin{table}[htbp!]
\centering
\begin{tabular}{|c|ccccccc|}
\hline
$\beta$ &$1.001d_p$&$\beta_\mathrm{best}$&$\beta_\mathrm{sug}$ &$2\beta_\mathrm{sug}$ &$5\beta_\mathrm{sug}$ &$10\beta_\mathrm{sug}$ &$100\beta_\mathrm{sug}$\\
\hline
LBFGS&971&475&518&574&580&594&1093\\ 
TNC&34&28&31&35&45&46&81\\ 
SympEigPen-B&6962&857&908&1186&1281&1321&2662\\ \hline
\end{tabular}
\caption{Iteration counts for sparse-add-low-rank $A$ and varying $\beta$.}
\label{tab:beta-sparse-add-low-rank}
\end{table}

\subsection{Convergence Rate Comparison}
We tested convergence rates for varying $n$ and $p$, using the wall-clock time as a consistent metric.

Figures \ref{fig:error-time-dense}--\ref{fig:error-time-sparse-add-low-rank} illustrate the relationship between Golub-Werman subspace errors and wall-clock time for all compared algorithms. A key observation is that all optimization methods under the \textit{TPM} framework (namely \texttt{L-BFGS-B}, \texttt{TNC}) and our proposed \textit{SympEigPen}, successfully achieve convergence. Notably, their convergence rates are significantly faster than that of the Riemannian optimization method across all test cases.

Among the \textit{TPM}-based approaches, \textit{SympEigPen}, our enhanced BB-step nonmonotonic gradient method, consistently demonstrates the best performance. It not only reaches a given error tolerance in the shortest time but also exhibits stable convergence behavior. This superior performance substantiates the effectiveness of our core strategy: adaptively adjusting the penalty parameters by leveraging the special structure of first-order stationary points during the minimization of the penalty function. This adaptive mechanism appears crucial for efficient optimization.

Furthermore, the advantage of \textit{SympEigPen} becomes more pronounced as the number of desired eigenvectors $p$ increases. It maintains a clear performance lead over both the Riemannian method and other unconstrained solvers (\texttt{L-BFGS-B}, \texttt{TNC}), suggesting that our method offers superior scalability for problems requiring larger subspaces.

\begin{figure}[htbp!]
\centering
\begin{subfigure}[b]{0.3\linewidth}
    \centering
    \includegraphics[width=\linewidth]{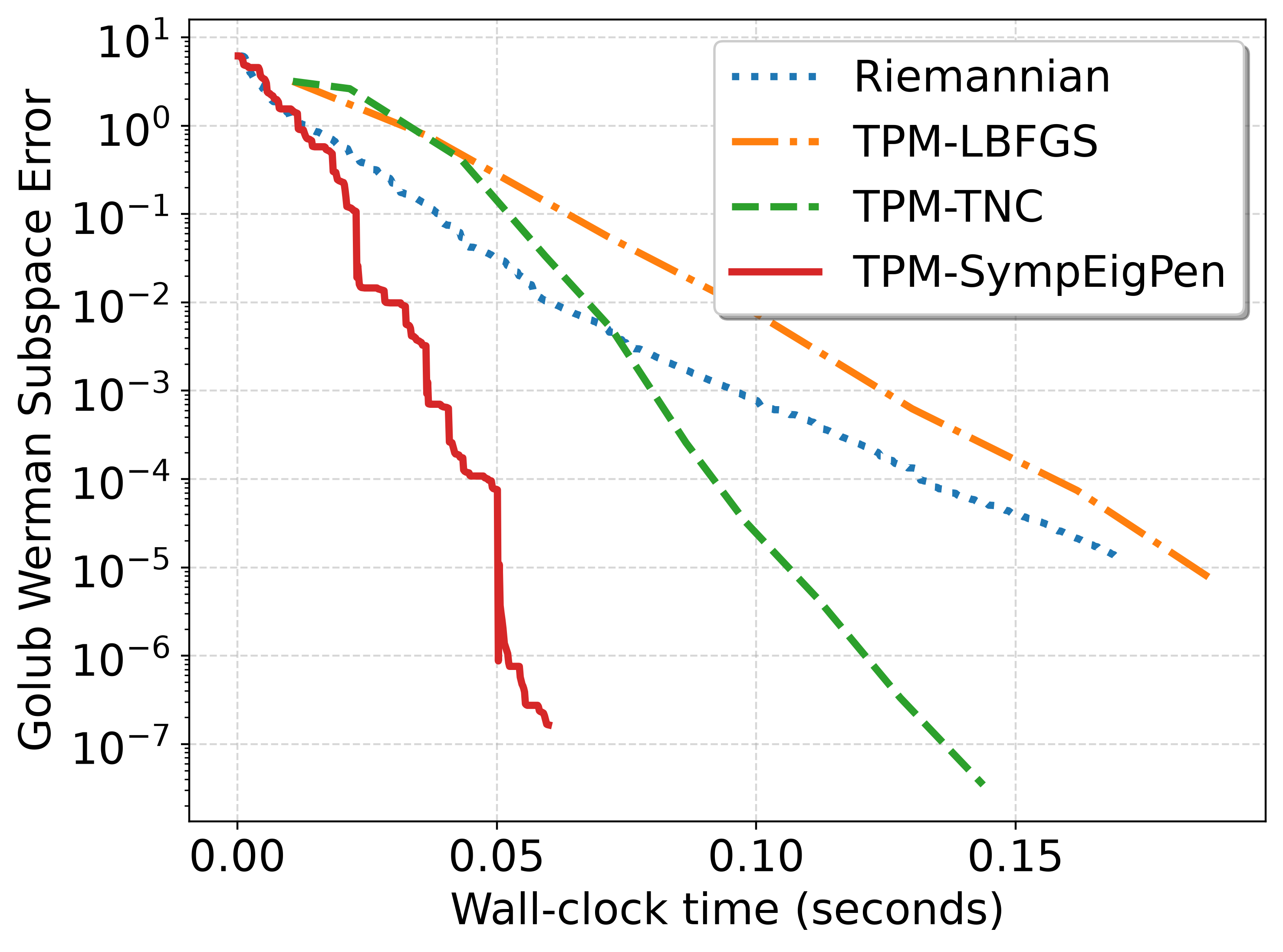}
    \caption{Dense $A$, $200 \times 200$, $p=10$}
    \label{fig:dense-sub11}
\end{subfigure}
\begin{subfigure}[b]{0.3\linewidth}
    \centering
    \includegraphics[width=\linewidth]{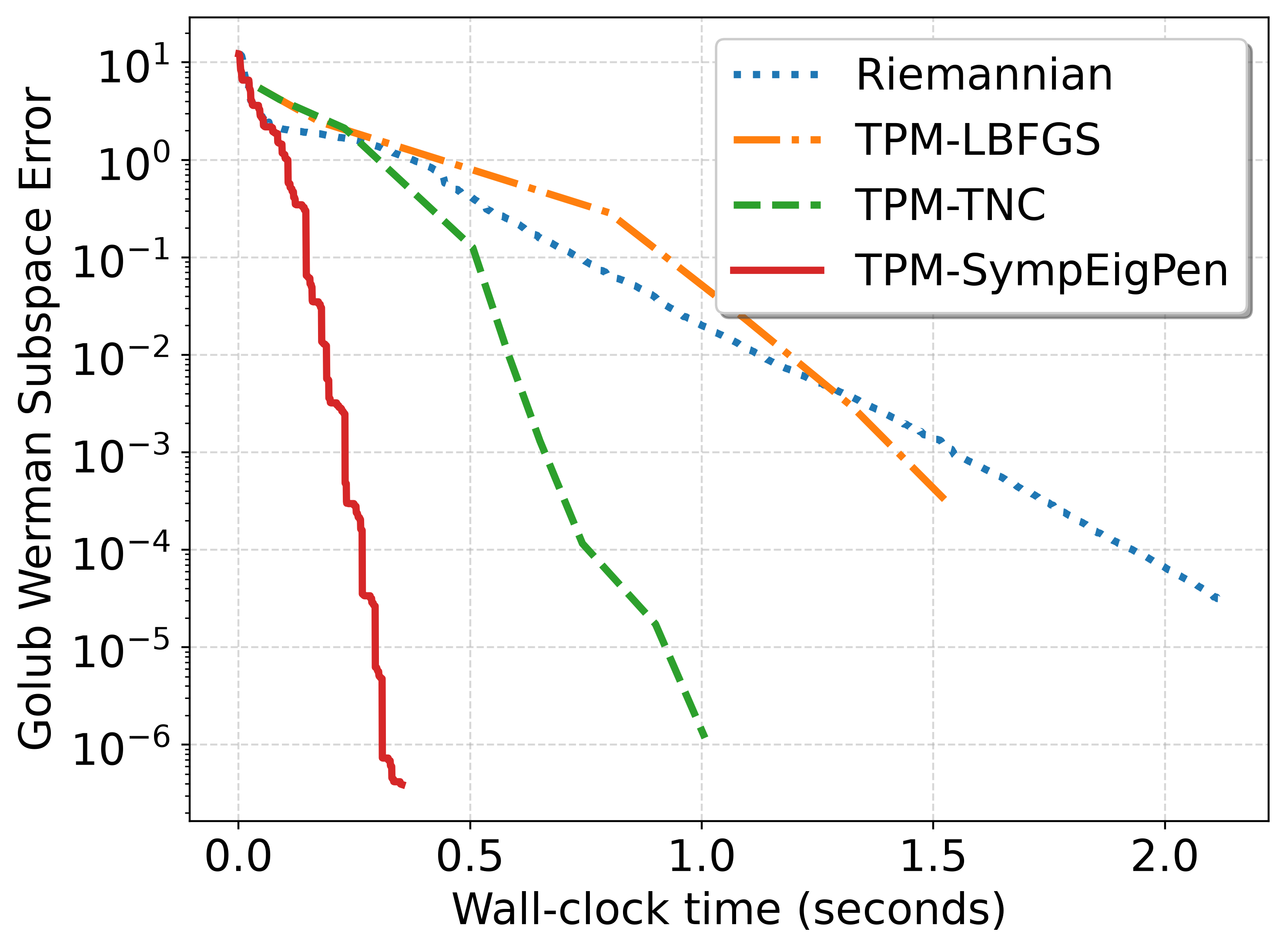}
    \caption{Dense $A$, $200 \times 200$, $p=50$}
    \label{fig:dense-sub12}
\end{subfigure}
\begin{subfigure}[b]{0.3\linewidth}
    \centering
    \includegraphics[width=\linewidth]{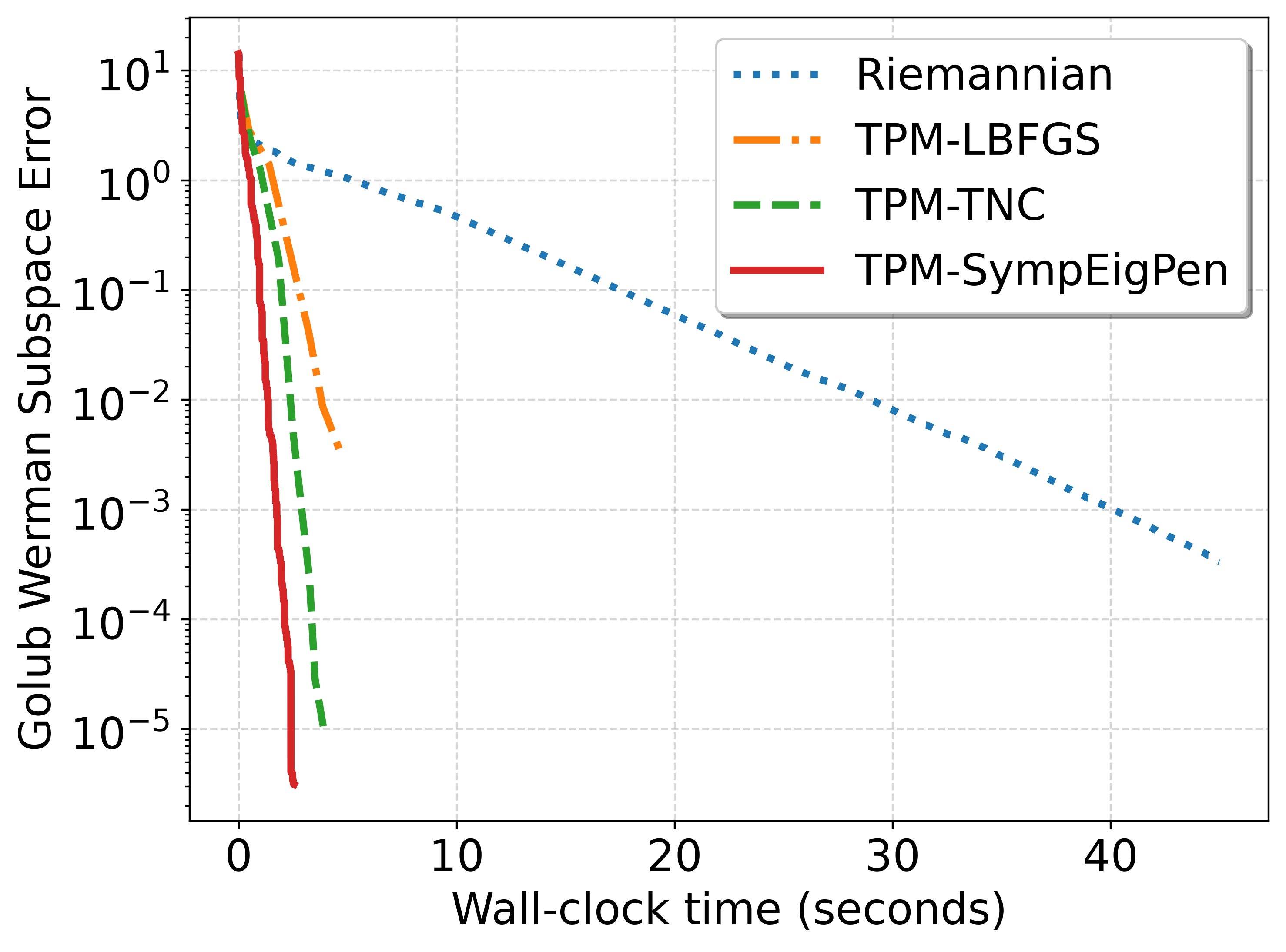}
    \caption{Dense $A$, $200 \times 200$, $p=100$}
    \label{fig:dense-sub13}
\end{subfigure}\\
\begin{subfigure}[b]{0.3\linewidth}
    \centering
    \includegraphics[width=\linewidth]{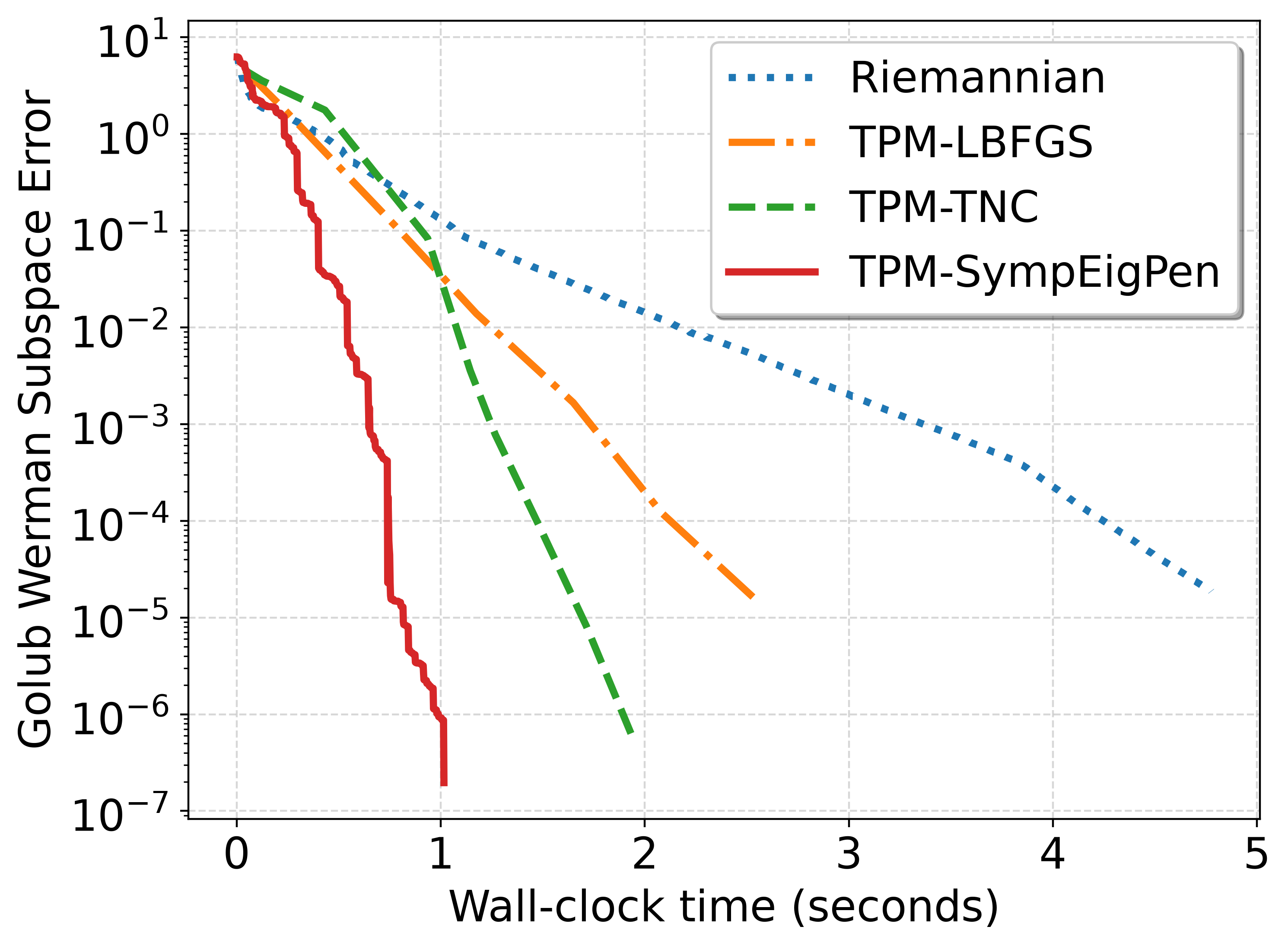}
    \caption{Dense $A$, $400 \times 400$, $p=10$}
    \label{fig:dense-sub21}
\end{subfigure}
\begin{subfigure}[b]{0.3\linewidth}
    \centering
    \includegraphics[width=\linewidth]{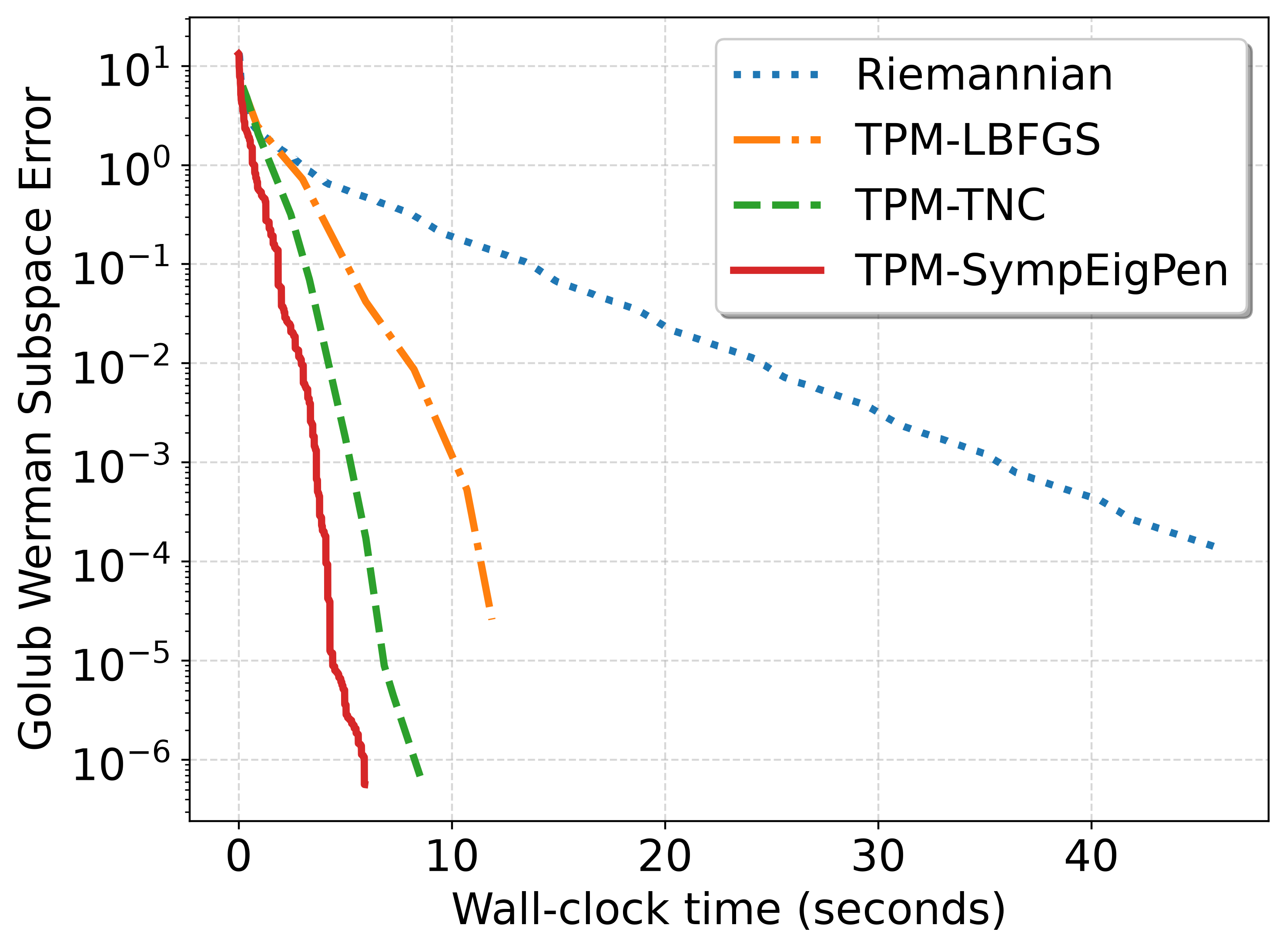}
    \caption{Dense $A$, $400 \times 400$, $p=50$}
    \label{fig:dense-sub22}
\end{subfigure}
\begin{subfigure}[b]{0.3\linewidth}
    \centering
    \includegraphics[width=\linewidth]{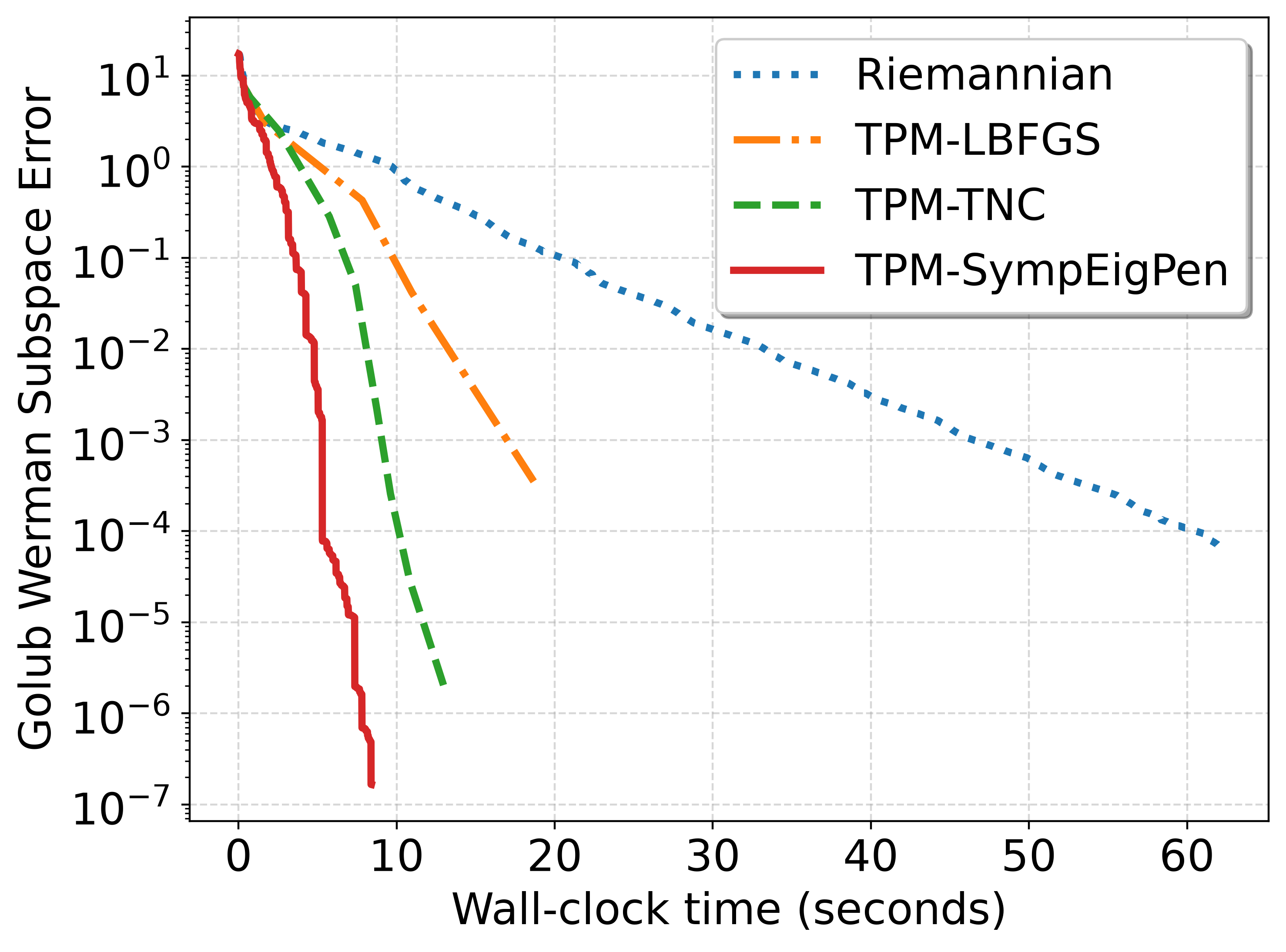}
    \caption{Dense $A$, $400 \times 400$, $p=100$}
    \label{fig:dense-sub23}
\end{subfigure}
\caption{Golub-Werman Subspace errors for dense $A$ (sizes $200\times200$, $\hat{d}_i=i$) and varying $p$.}
\label{fig:error-time-dense}
\end{figure}

\begin{figure}[htbp!]
\centering
\begin{subfigure}[b]{0.3\linewidth}
    \centering
    \includegraphics[width=\linewidth]{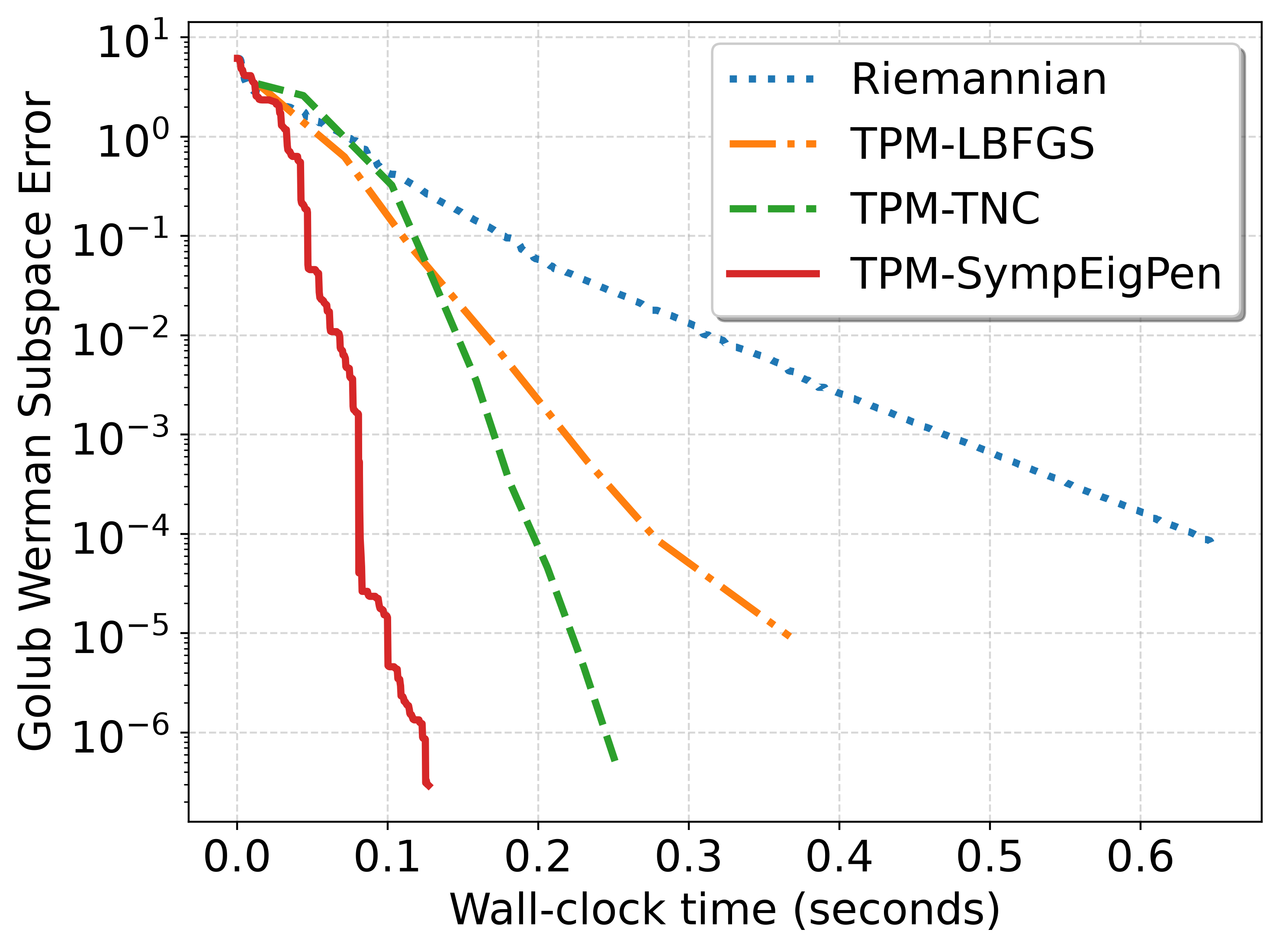}
    \caption{Sparse $A$, $200 \times 200$, $p=10$}
    \label{fig:sparse-sub11}
\end{subfigure}
\begin{subfigure}[b]{0.3\linewidth}
    \centering
    \includegraphics[width=\linewidth]{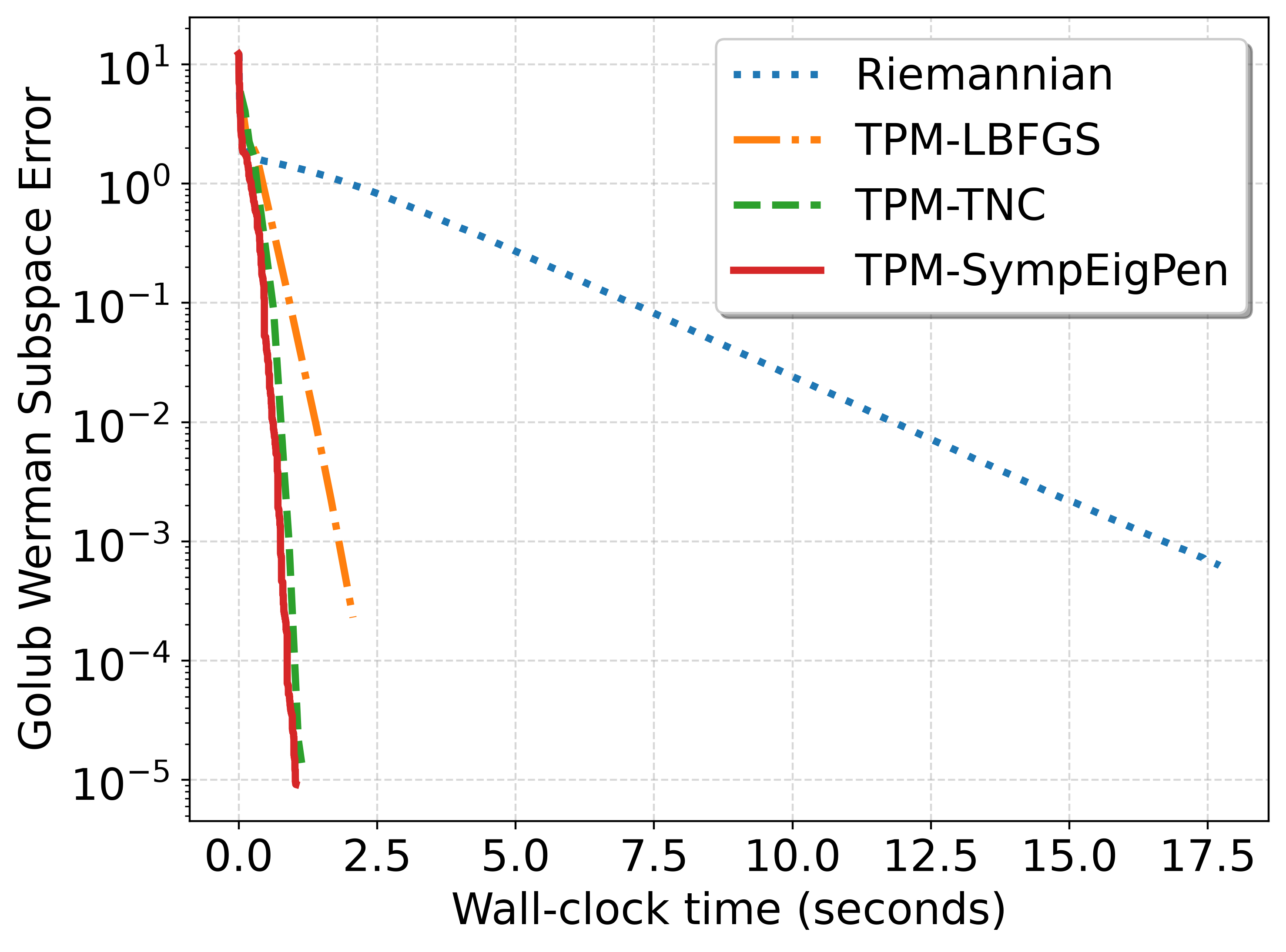}
    \caption{Sparse $A$, $200 \times 200$, $p=50$}
    \label{fig:sparse-sub12}
\end{subfigure}
\begin{subfigure}[b]{0.3\linewidth}
    \centering
    \includegraphics[width=\linewidth]{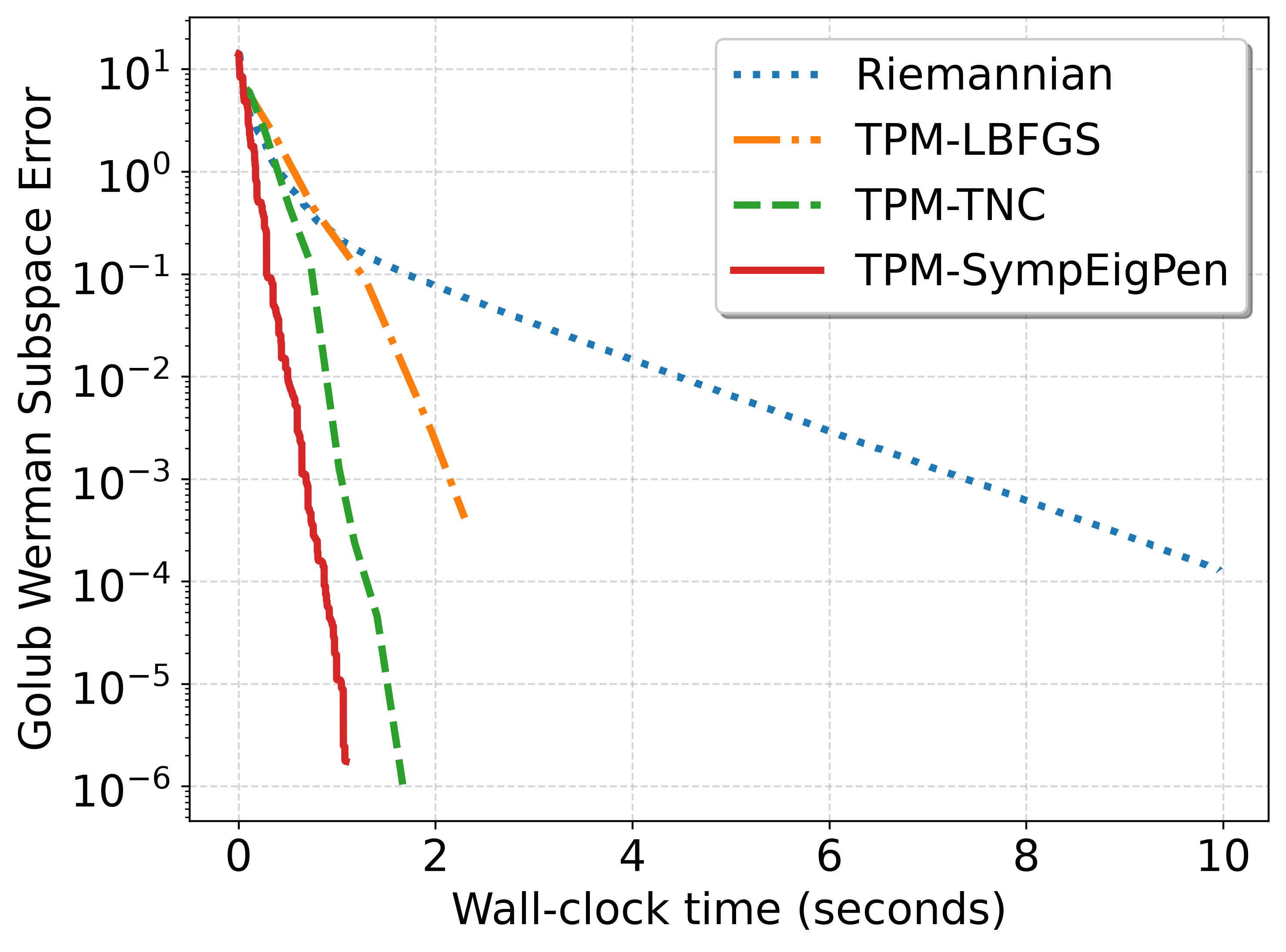}
    \caption{Sparse $A$, $200 \times 200$, $p=100$}
    \label{fig:sparse-sub13}
\end{subfigure}\\
\begin{subfigure}[b]{0.3\linewidth}
    \centering
    \includegraphics[width=\linewidth]{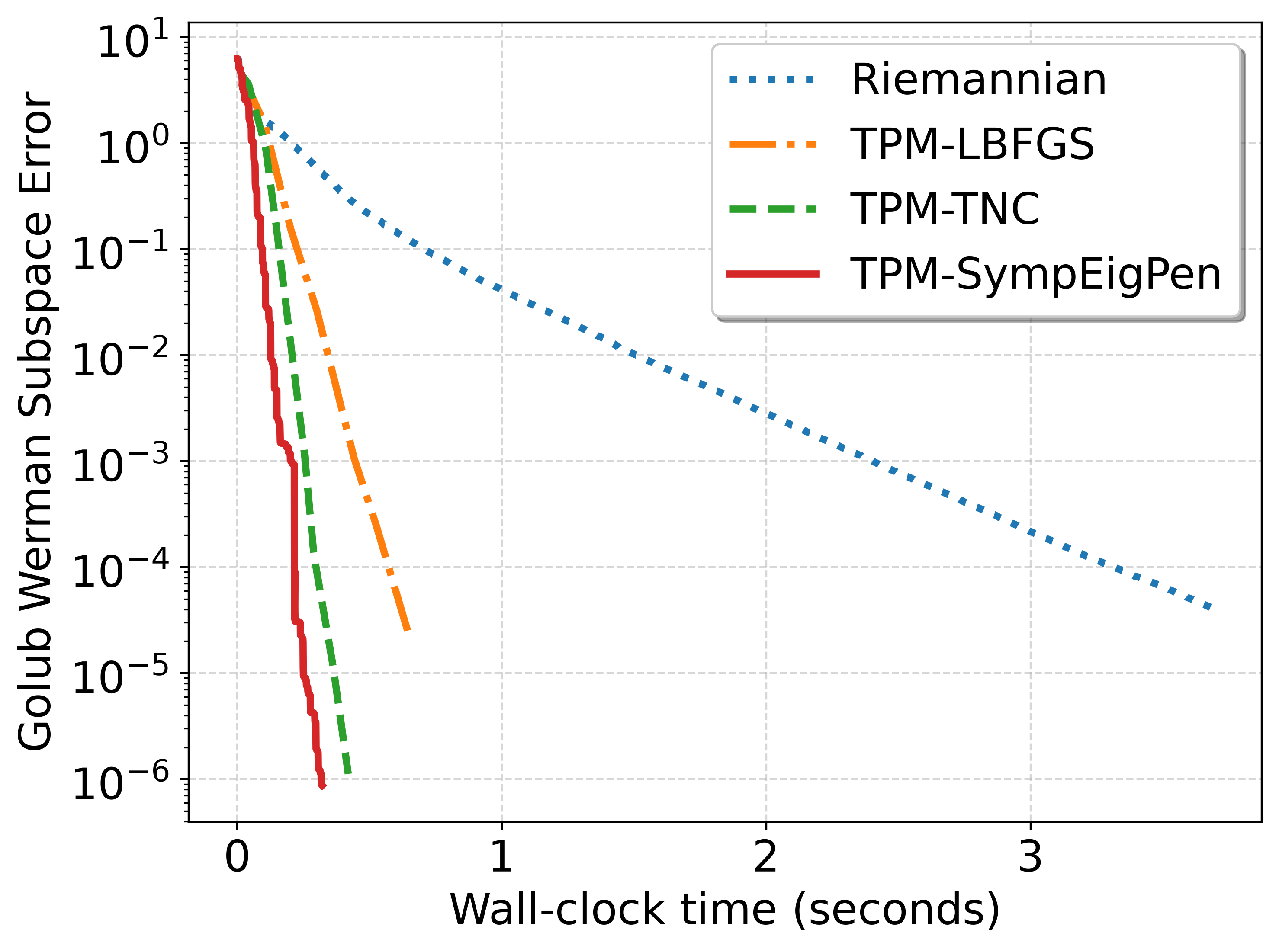}
    \caption{Sparse $A$, $400 \times 400$, $p=10$}
    \label{fig:sparse-sub21}
\end{subfigure}
\begin{subfigure}[b]{0.3\linewidth}
    \centering
    \includegraphics[width=\linewidth]{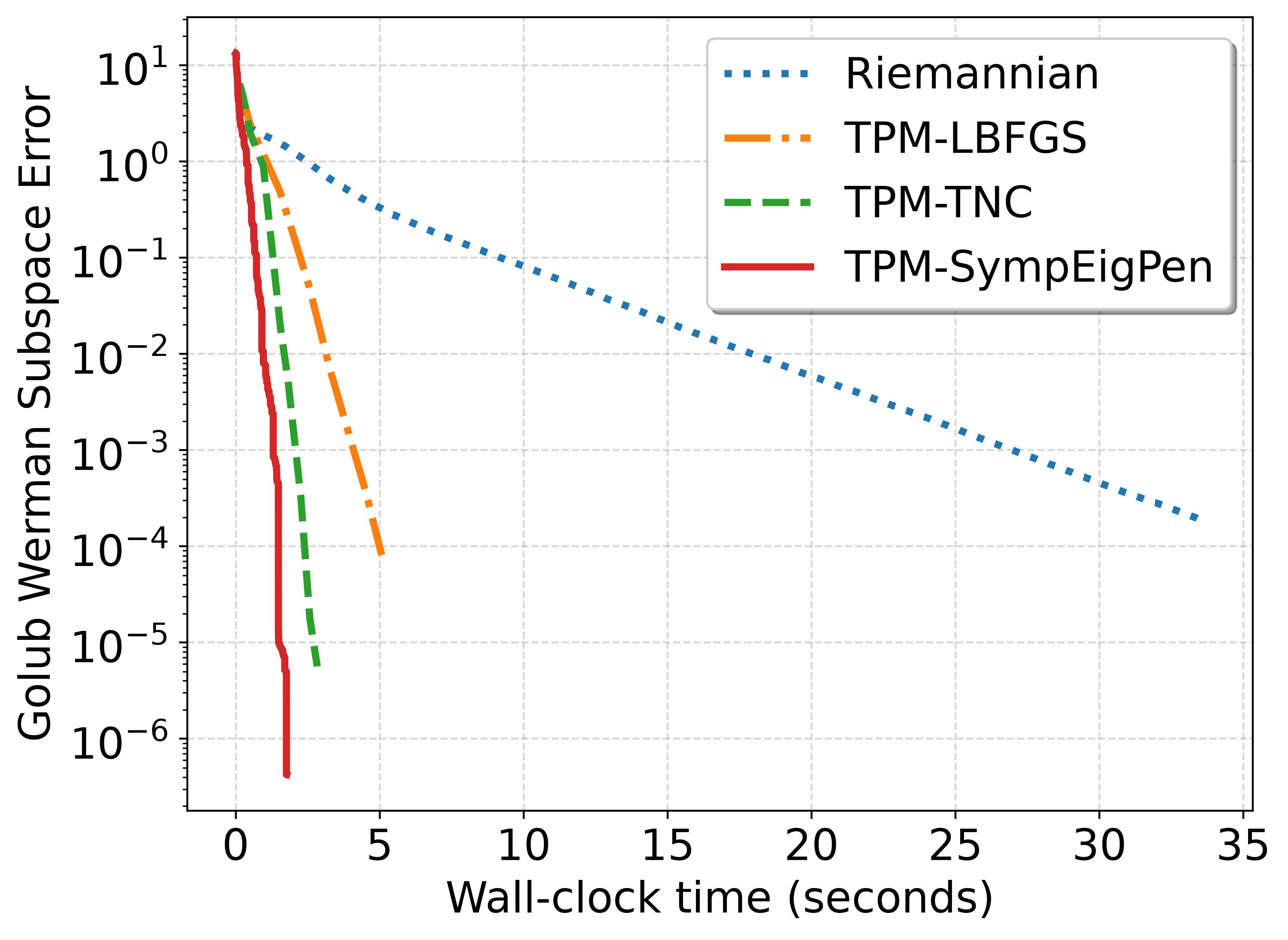}
    \caption{Sparse $A$, $400 \times 400$, $p=50$}
    \label{fig:sparse-sub22}
\end{subfigure}
\begin{subfigure}[b]{0.3\linewidth}
    \centering
    \includegraphics[width=\linewidth]{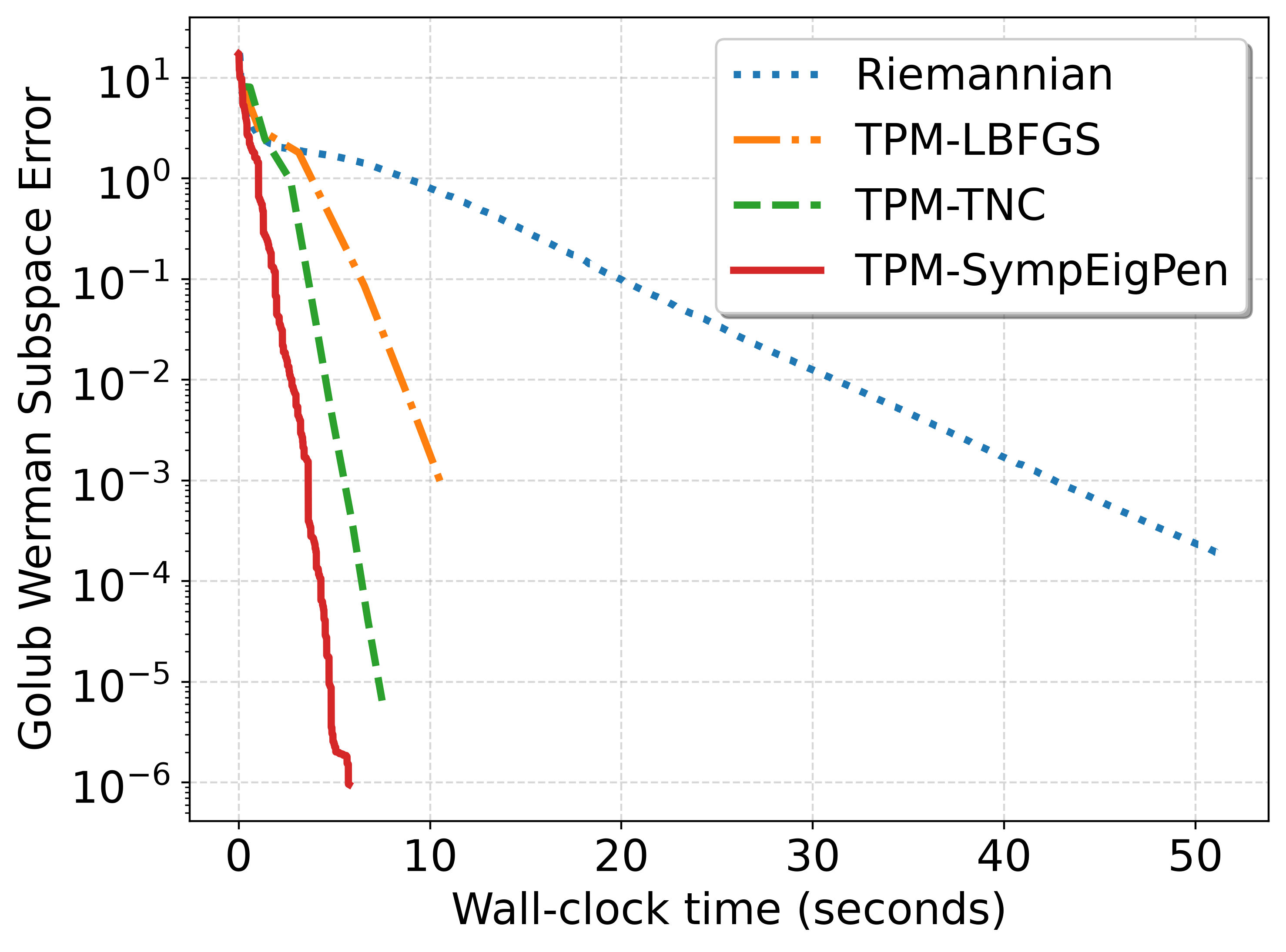}
    \caption{Sparse $A$, $400 \times 400$, $p=100$}
    \label{fig:sparse-sub23}
\end{subfigure}
\caption{Golub-Werman Subspace errors for sparse $A$ (sizes $200\times200$) and varying $p$.}
\label{fig:error-time-sparse}
\end{figure}

\begin{figure}[htbp!]
\centering
\begin{subfigure}[b]{0.3\linewidth}
    \centering
    \includegraphics[width=\linewidth]{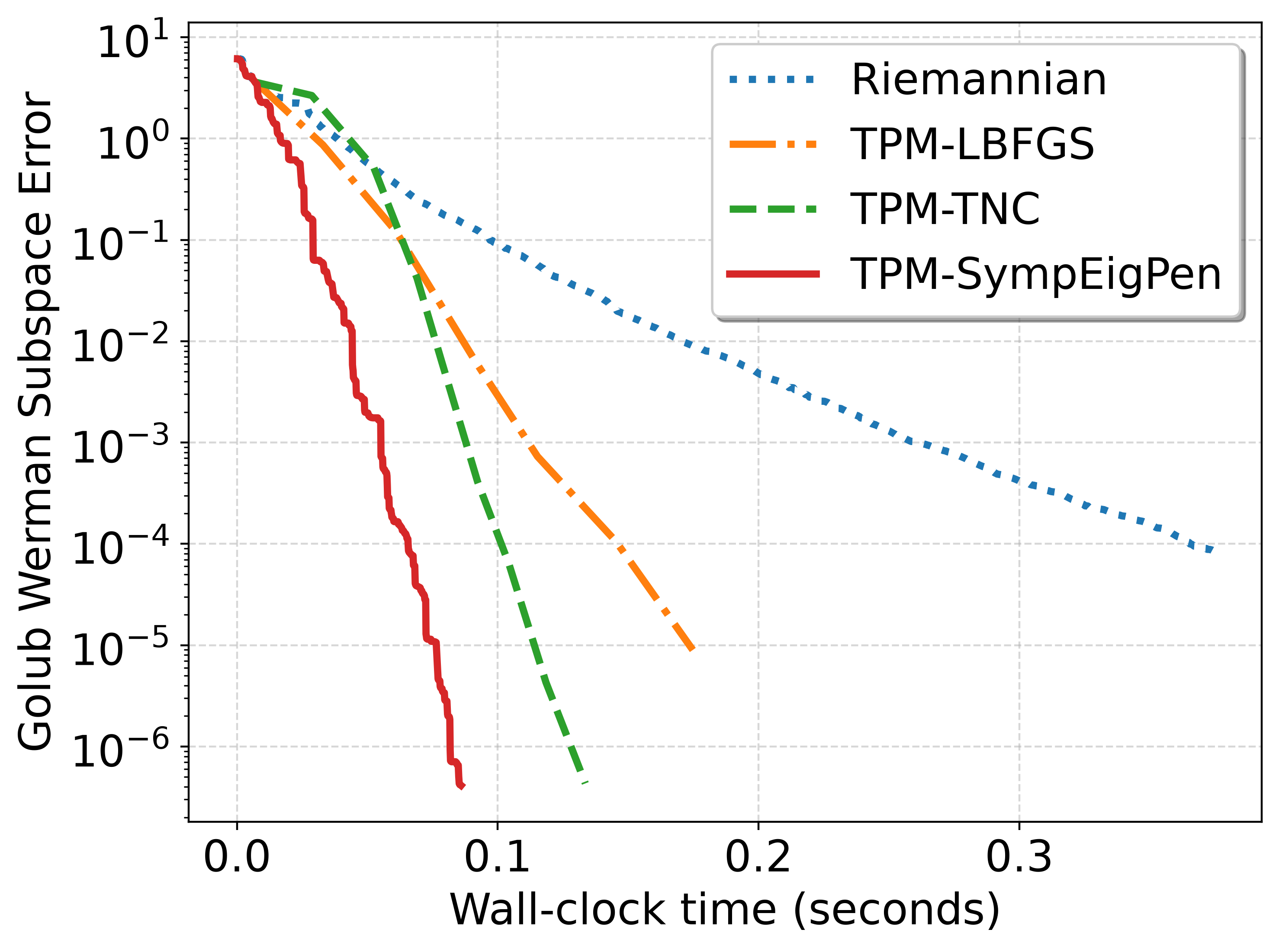}
    \caption{Sparse-add-low-rank $A$, $200 \times 200$, $p=10$}
    \label{fig:slr-sub11}
\end{subfigure}
\begin{subfigure}[b]{0.3\linewidth}
    \centering
    \includegraphics[width=\linewidth]{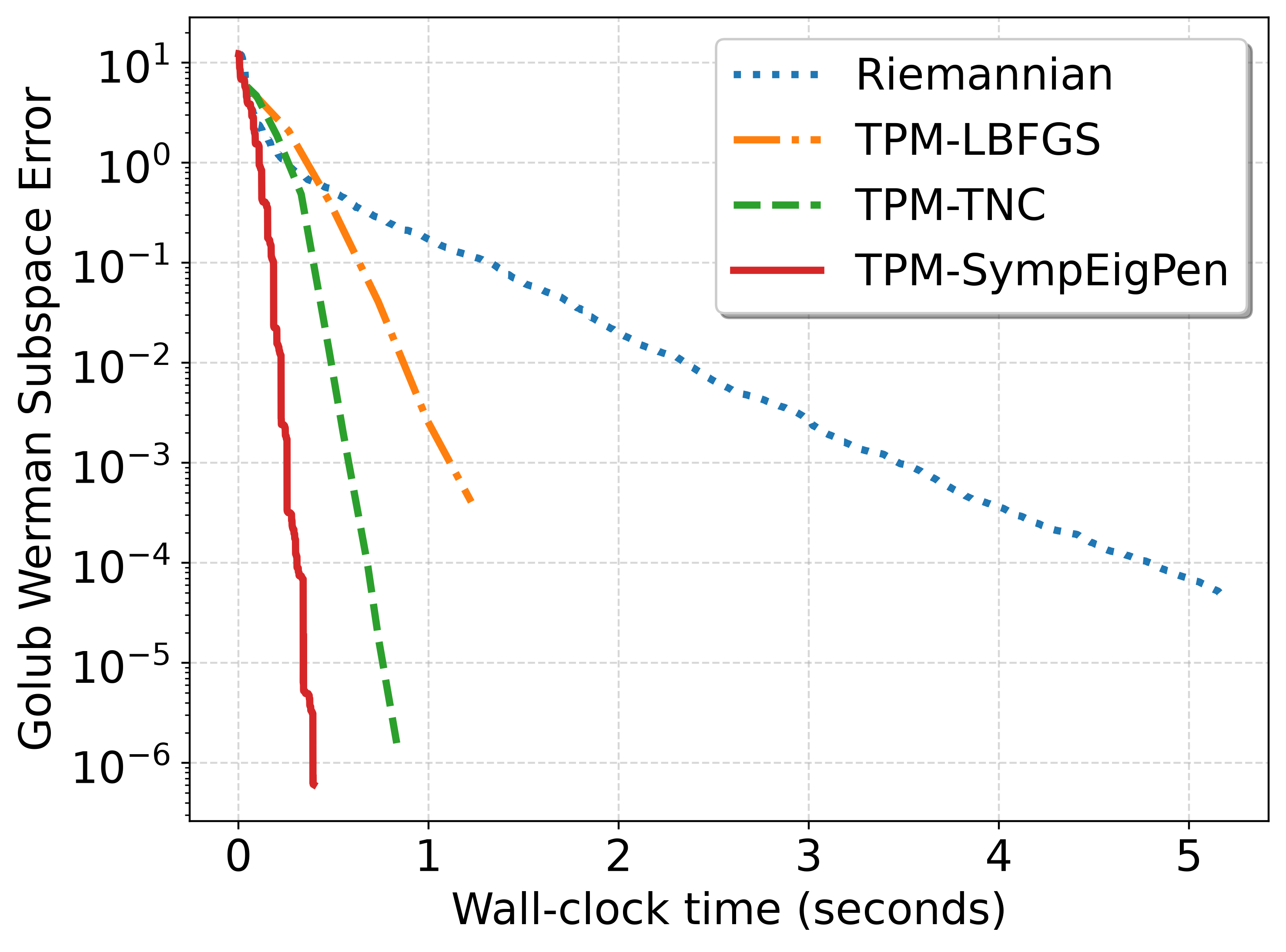}
    \caption{Sparse-add-low-rank $A$, $200 \times 200$, $p=50$}
    \label{fig:slr-sub12}
\end{subfigure}
\begin{subfigure}[b]{0.3\linewidth}
    \centering
    \includegraphics[width=\linewidth]{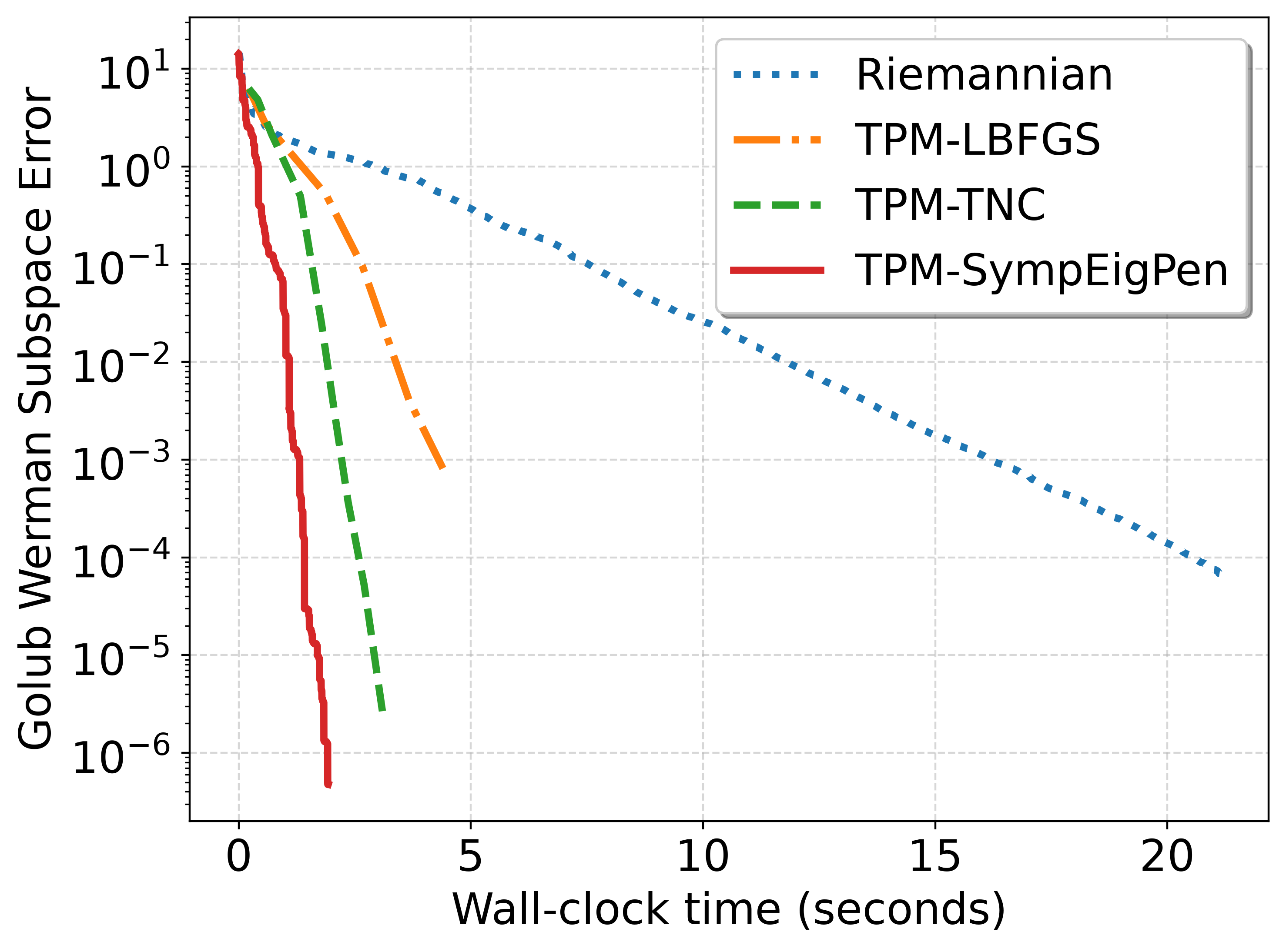}
    \caption{Sparse-add-low-rank $A$, $200 \times 200$, $p=100$}
    \label{fig:slr-sub13}
\end{subfigure}\\
\begin{subfigure}[b]{0.3\linewidth}
    \centering
    \includegraphics[width=\linewidth]{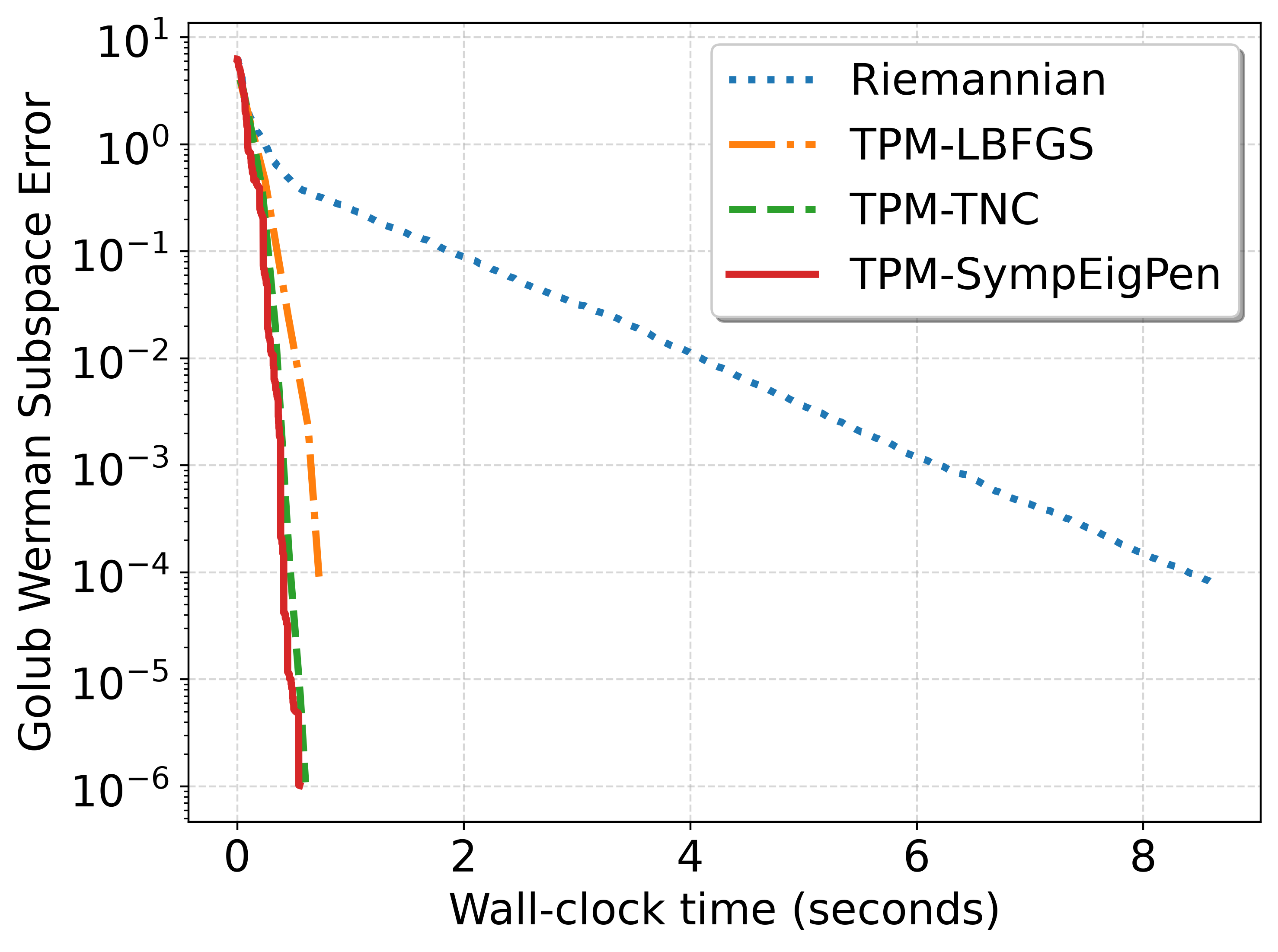}
    \caption{Sparse-add-low-rank $A$, $400 \times 400$, $p=10$}
    \label{fig:slr-sub21}
\end{subfigure}
\begin{subfigure}[b]{0.3\linewidth}
    \centering
    \includegraphics[width=\linewidth]{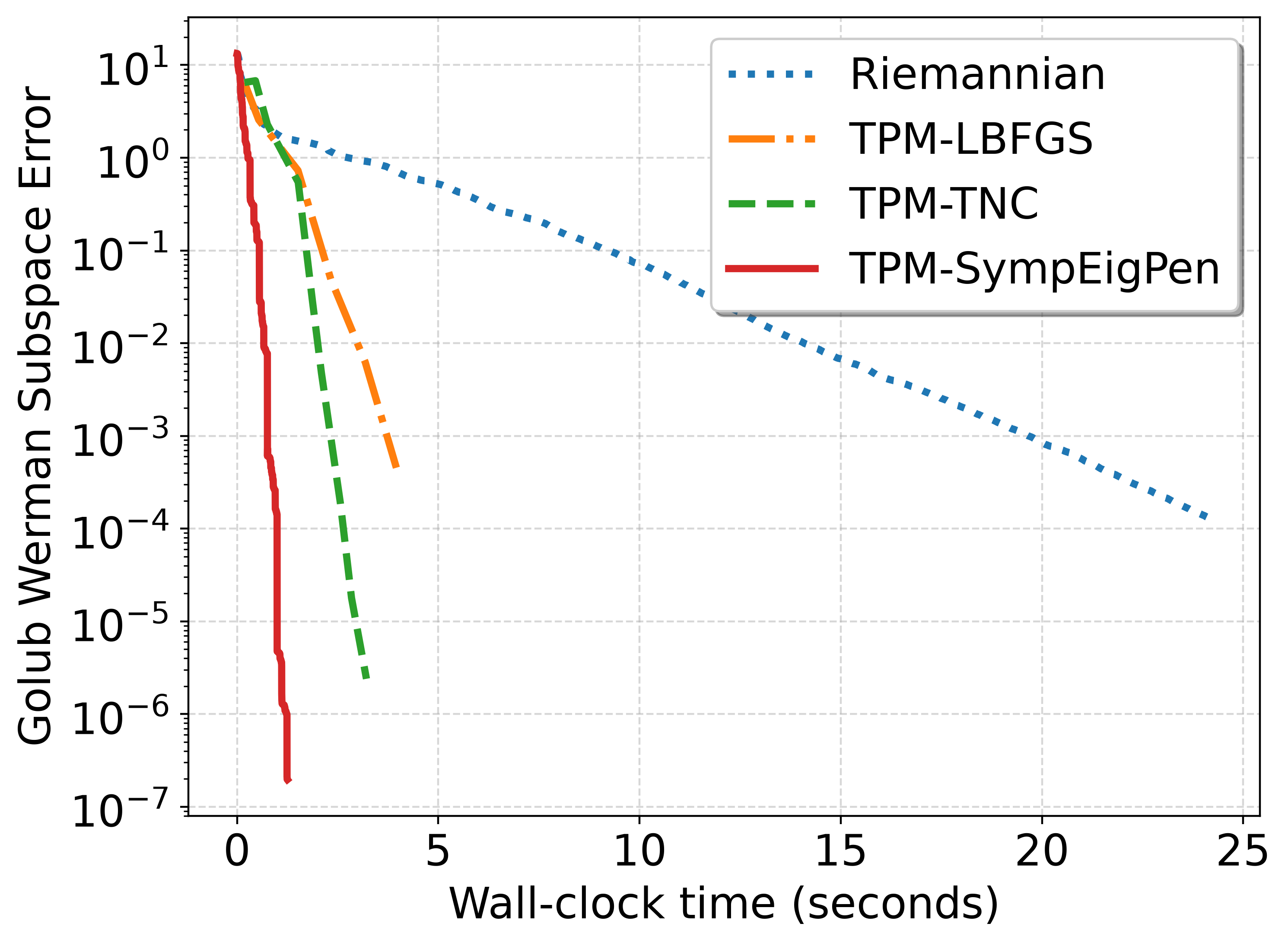}
    \caption{Sparse-add-low-rank $A$, $400 \times 400$, $p=50$}
    \label{fig:slr-sub22}
\end{subfigure}
\begin{subfigure}[b]{0.3\linewidth}
    \centering
    \includegraphics[width=\linewidth]{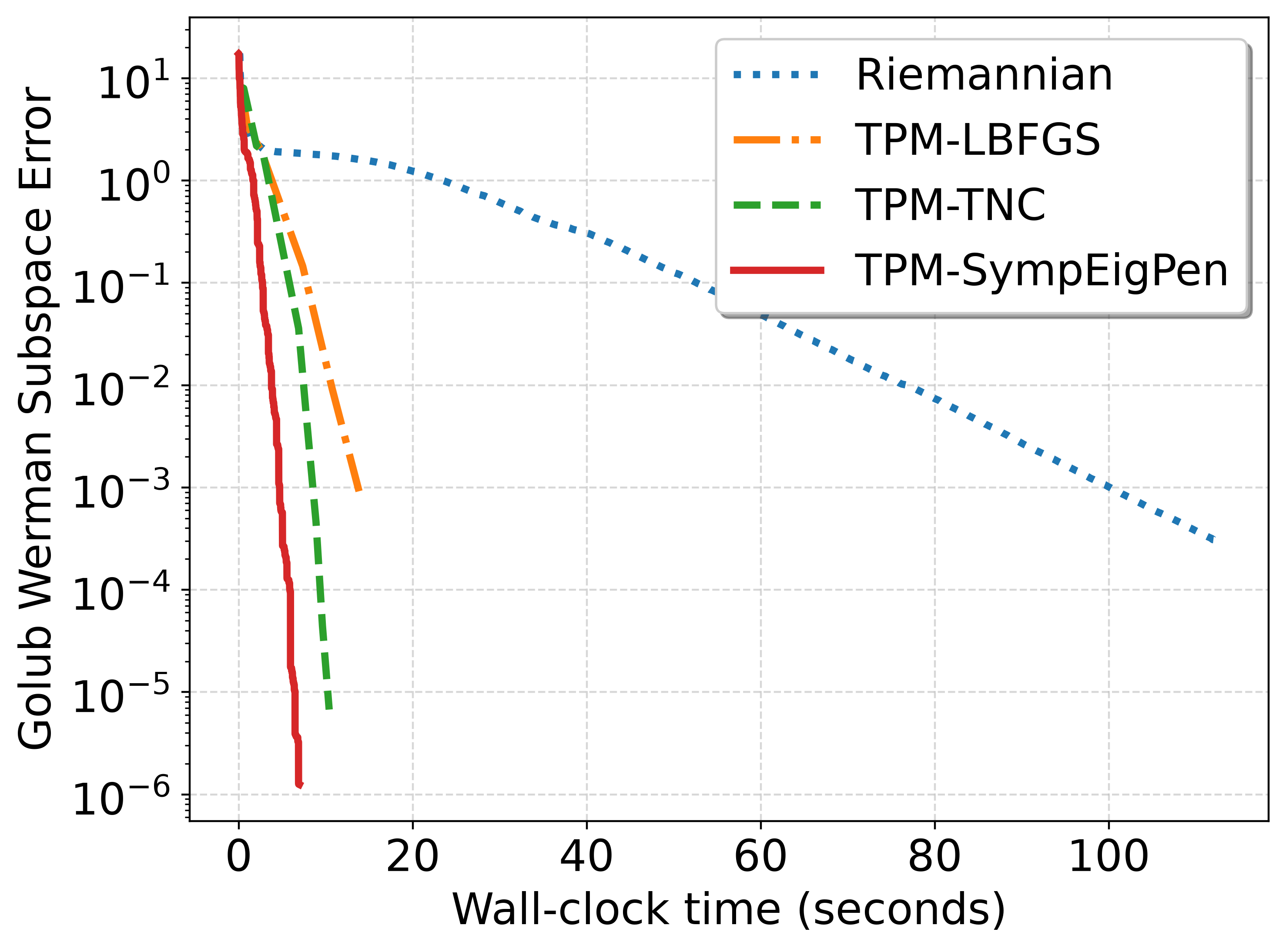}
    \caption{Sparse-add-low-rank $A$, $400 \times 400$, $p=100$}
    \label{fig:slr-sub23}
\end{subfigure}
\caption{Golub-Werman Subspace errors for Sparse-add-low-rank $A$ (sizes $200\times200$) and varying $p$.}
\label{fig:error-time-sparse-add-low-rank}
\end{figure}

Based on the comprehensive data in Table \ref{tab:time-err-res}, the numerical experiments clearly show that the \textit{TPM} framework is more efficient than the \texttt{eigs} function and the Riemannian method (RM). This holds true for different matrix types (dense, sparse, sparse-add-low-rank) and as the matrix height $n$ increases from $200$ to $51200$.

Among the tested optimizers within the \textit{TPM} framework, \textit{SympEigPen} (BB) is consistently the fastest. For the largest sparse matrix problem ($n=51200$), \textit{SympEigPen} only takes $132.90$ seconds, while the \texttt{eigs} method failed to finish within an hour. The advantage of \textit{TPM}-based methods becomes more significant as the problem size grows, which demonstrates the better scalability of \textit{TPM}-based algorithms for large-scale problems.

In terms of accuracy, all methods, including the different \textit{TPM} optimizers (\texttt{L-BFGS-B}, \texttt{TNC}, \textit{SympEigPen}), achieved low Golub-Werman subspace errors ($\mathcal{E}$) and residues ($\mathcal{R}$), typically ranging from $10^{-5}$ to $10^{-10}$. The \textit{SympEigPen} optimizer not only provided high computational speed but also maintained excellent solution quality, with errors and residues comparable to or even lower than the other approaches.

\begin{table}[htbp!]
\centering
\resizebox{\linewidth}{!}{
\begin{tabular}{|c|c|ccc|ccc|ccc|ccc|}
\hline
\multicolumn{14}{|c|}{Dense $A$}\\
\hline
$n$ & $t_{\text{eigs}}$& $t_{\text{RM}}$ & $\mathcal{E}_{\text{RM}}$ & $\mathcal{R}_{\text{RM}}$ & $t_{\text{LBFGS}}$ & $\mathcal{E}_{\text{LBFGS}}$ & $\mathcal{R}_{\text{LBFGS}}$ & $t_{\text{TNC}}$ & $\mathcal{E}_{\text{TNC}}$ & $\mathcal{R}_{\text{TNC}}$ & $t_{\text{BB}}$ & $\mathcal{E}_{\text{BB}}$ & $\mathcal{R}_{\text{BB}}$ \\
\hline
200&0.16s&0.11s&1.27e-5&5.23e-7&0.13s&4.70e-6&2.57e-7&0.09s&3.54e-8&6.78e-9&0.03s&1.65e-7&2.40e-9\\
400&0.22s&0.78s&1.85e-5&2.70e-7&0.62s&1.08e-5&3.66e-7&0.37s&1.79e-7&1.26e-8&0.11s&1.96e-7&1.64e-9\\
800&2.76s&1.88s&1.38e-5&3.22e-7&1.86s&1.16e-5&3.16e-7&1.21s&7.44e-8&8.61e-9&0.40s&1.43e-7&1.17e-9\\
1600&19.69s&25.62s&9.06e-5&5.48e-7&12.52s&1.71e-5&4.06e-7&8.45s&3.94e-7&1.07e-8&3.45s&2.92e-7&7.95e-10\\
3200&115.57s&73.09s&3.26e-5&5.18e-7&50.39s&1.31e-5&4.82e-7&32.18s&3.39e-7&1.71e-8&11.05s&4.46e-8&4.54e-10\\
6400&814.72s&642.00s&4.33e-5&3.43e-7&398.63s&3.93e-5&7.81e-7&209.37s&4.97e-7&2.88e-8&81.56s&7.86e-8&4.04e-10\\
12800&--&--&--&--&1439.96s&5.47e-5&8.07e-7&947.89s&2.25e-6&4.09e-8&469.16s&8.05e-8&2.45e-10\\
\hline
\multicolumn{14}{|c|}{Sparse $A$}\\
\hline
$n$ & $t_{\text{eigs}}$& $t_{\text{RM}}$ & $\mathcal{E}_{\text{RM}}$ & $\mathcal{R}_{\text{RM}}$ & $t_{\text{LBFGS}}$ & $\mathcal{E}_{\text{LBFGS}}$ & $\mathcal{R}_{\text{LBFGS}}$ & $t_{\text{TNC}}$ & $\mathcal{E}_{\text{TNC}}$ & $\mathcal{R}_{\text{TNC}}$ & $t_{\text{BB}}$ & $\mathcal{E}_{\text{BB}}$ & $\mathcal{R}_{\text{BB}}$ \\
\hline
200&0.41s&0.28s&7.86e-5&2.64e-6&0.15s&6.21e-6&2.75e-7&0.10s&1.15e-7&7.71e-9&0.03s&2.94e-7&3.58e-9\\
400&0.12s&1.36s&4.09e-5&2.25e-7&0.56s&1.09e-5&2.41e-7&0.31s&2.45e-7&7.64e-9&0.10s&8.73e-7&3.84e-9\\
800&0.63s&2.00s&1.64e-5&1.80e-7&1.27s&5.55e-6&2.70e-7&0.78s&5.80e-8&4.51e-9&0.19s&2.11e-7&2.00e-9\\
1600&3.50s&8.13s&3.57e-5&3.54e-7&4.14s&1.73e-5&2.87e-7&2.25s&1.64e-7&5.68e-9&0.71s&3.01e-7&1.56e-9\\
3200&73.09s&37.22s&5.66e-5&4.49e-7&13.74s&2.00e-5&4.09e-7&7.19s&2.47e-7&8.36e-9&2.08s&1.68e-7&6.57e-10\\
6400&341.29s&82.37s&6.27e-5&6.47e-7&40.59s&1.85e-5&3.93e-7&18.35s&4.79e-7&1.67e-8&6.38s&1.06e-7&6.83e-10\\
12800&536.81s&522.37s&1.82e-4&8.89e-7&124.75s&2.86e-5&5.13e-7&69.07s&2.08e-6&2.33e-8&28.65s&2.00e-7&4.57e-10\\
25600&1343.26s&2037.10s&2.28e-5&8.32e-7&325.21s&2.02e-5&8.68e-7&203.92s&3.96e-8&4.02e-9&62.28s&4.24e-7&5.11e-9\\
51200&--&--&--&--&1012.85s&6.60e-5&2.72e-7&592.31s&5.52e-8&2.20e-9&192.85s&5.16e-7&2.48e-9\\
\hline
\multicolumn{14}{|c|}{Sparse-add-low-rank $A$}\\
\hline
$n$ & $t_{\text{eigs}}$& $t_{\text{RM}}$ & $\mathcal{E}_{\text{RM}}$ & $\mathcal{R}_{\text{RM}}$ & $t_{\text{LBFGS}}$ & $\mathcal{E}_{\text{LBFGS}}$ & $\mathcal{R}_{\text{LBFGS}}$ & $t_{\text{TNC}}$ & $\mathcal{E}_{\text{TNC}}$ & $\mathcal{R}_{\text{TNC}}$ & $t_{\text{BB}}$ & $\mathcal{E}_{\text{BB}}$ & $\mathcal{R}_{\text{BB}}$ \\
\hline
200&0.23s&0.39s&8.45e-5&2.32e-6&0.22s&5.41e-6&2.30e-7&0.16s&1.07e-7&6.15e-9&0.07s&4.13e-7&3.59e-9\\
400&0.15s&5.26s&7.97e-5&2.58e-7&0.92s&1.91e-5&4.99e-7&0.56s&2.10e-7&3.28e-9&0.23s&1.02e-6&2.19e-9\\
800&4.83s&4.53s&2.08e-5&3.77e-7&2.12s&6.72e-6&2.80e-7&1.24s&2.60e-7&9.34e-9&0.45s&1.48e-7&1.17e-9\\
1600&6.72s&13.11s&3.34e-5&5.10e-7&6.50s&9.78e-6&3.77e-7&3.52s&1.69e-7&6.26e-9&1.12s&6.91e-8&9.24e-10\\
3200&167.59s&64.98s&7.27e-5&6.24e-7&19.07s&4.29e-5&7.14e-7&11.36s&3.12e-7&1.08e-8&3.66s&1.36e-7&6.86e-10\\
6400&226.09s&162.28s&4.66e-5&5.43e-7&57.64s&1.92e-5&3.82e-7&28.21s&3.34e-7&1.95e-8&10.14s&7.65e-8&4.67e-10\\
12800&592.62s&594.24s&2.37e-5&6.42e-7&165.22s&2.28e-5&4.23e-7&78.76s&5.68e-7&2.95e-8&32.96s&4.80e-8&8.21e-10\\
25600&1692.12s&2261.22s&4.36e-5&5.17e-7&382.11s&1.79e-5&6.26e-7&263.14s&2.29e-8&3.13e-9&76.42s&2.20e-7&3.45e-9\\
51200&--&--&--&--&1321.22s&3.23e-5&1.58e-7&703.26s&4.23e-8&4.32e-9&221.29s&4.62e-7&3.16e-9\\
\hline
\end{tabular}
}
\caption{Wall-clock time, Golub-Werman subspace errors and residues for three types of $A$ (varying $n$, $p=10$).}
\label{tab:time-err-res}
\end{table}

\section{Conclusion}
\label{sec:con}
In this paper, we consider the symplectic eigenvalue problem for symmetric positive-definite matrices through its trace minimization formulation \eqref{Prob_Ori}, where the constraints lead to significant challenges to the development of efficient algorithms. In our proposed approach, we reformulate the trace minimization problem \eqref{Prob_Ori} into an unconstrained one \eqref{Prob_Pen}. We demonstrate that \eqref{Prob_Pen} is equivalent to the original problem \eqref{Prob_Ori} with a sufficiently large but finite $\beta$, in the sense that any second-order stationary point of \eqref{Prob_Pen} corresponds to a global minimizer of \eqref{Prob_Pen}. To solve \eqref{Prob_Pen} efficiently, we develop SympEigPen, a non-monotonic gradient method with a Barzilai-Borwein step size. Preliminary experiments demonstrate that SympEigPen outperforms existing methods in convergence rate and efficiency.

Our proposed penalty function \eqref{Prob_Pen} opens several promising avenues for future research. The established equivalence between \eqref{Prob_Ori} and \eqref{Prob_Pen} provides a foundation for developing efficient methods for symplectic eigenvalue problems in stochastic or decentralized settings. Under these settings, the data matrix $A$ is either formulated as an expectation or partitioned among agents in a network with local communication. Since geometric computations on the symplectic Stiefel manifold are computationally expensive, existing Riemannian optimization approaches cannot efficiently solve symplectic eigenvalue problems in these settings. Consequently, applying the quadratic penalty \eqref{Prob_Pen} to stochastic or decentralized settings remains a valuable direction for future work.

\section*{Declarations}

\subsection*{Competing Interests}
The authors have no competing interests to declare that are relevant to the content of this article.

\subsection*{Funding}
No funding was received to assist with the preparation of this manuscript.

\subsection*{Code Availability}
The code used to generate the numerical results is available at \url{https://github.com/wangjiaqi-amss/sympeigpen.git}.

\section*{Acknowledgements}
The authors sincerely thank Prof. Hemant K. Mishra for his valuable suggestions and comments.

\bibliographystyle{plain}
\bibliography{ref}

\end{document}